\documentclass[12pt,reqno]{amsart}
\usepackage{latexsym,amsmath,mathtools}
\usepackage{amsfonts}
\usepackage{amssymb}
\usepackage{latexsym}
\usepackage{enumerate}
\usepackage[T1]{fontenc}
\usepackage{fourier}
\usepackage{bbm}
\usepackage{color}
\usepackage{xcolor}
\usepackage{tikz}
\usepackage{fullpage}
\newtheorem{Theorem}{Theorem}[section]
\newtheorem{Lemma}[Theorem]{Lemma}
\newtheorem{Corollary}[Theorem]{Corollary}
\newtheorem{Definition}[Theorem]{Definition}

\newtheorem{Proposition}[Theorem]{Proposition}

\theoremstyle{remark} 
\theoremstyle{definition} 
\allowdisplaybreaks

\begin{document}
\title{The genus of the Erd\H{o}s-R\'enyi random graph \\ and the fragile genus property}
\author{Chris Dowden$^{*}$, Mihyun Kang$^{*}$, and Michael Krivelevich$^{\ddagger}$ \\ \\
\today}
\thanks{$^{*}$ 
Institute of Discrete Mathematics, 
Graz University of Technology, 
Steyrergasse 30,
8010 Graz,
Austria,  
{\tt \{dowden,kang\}@math.tugraz.at}.
Supported by Austrian Science Fund (FWF): P27290 and W1230II\phantom{}}
\thanks{$^{\ddagger}$ 
School of Mathematical Sciences, 
Sackler Faculty of Exact Sciences, 
Tel Aviv University,  
Tel Aviv 6997801,
Israel, 
{\tt krivelev@post.tau.ac.il}.
Supported in part by USA-Israel BSF grant 2014361, and by grant 1261/17 from Israel Science Foundation.}
\thanks{An extended abstract of this paper has been published in the proceedings of the 29th International Conference on Probabilistic, Combinatorial and Asymptotic Methods for the Analysis of Algorithms (AofA 2018),
Leibniz International Proceedings in Informatics (LIPIcs) 110 (2018), article number 17, 13 pages.}

\setlength{\unitlength}{1cm}

\begin{abstract}
We investigate the genus $g(n,m)$ of the Erd\H{o}s-R\'enyi random graph $G(n,m)$,
providing a thorough description of how this relates to the function $m=m(n)$,
and finding that there is different behaviour depending on which `region' $m$ falls into.

Results already exist for $m \le \frac{n}{2} + O(n^{2/3})$
and $m = \omega \left( n^{1+\frac{1}{j}} \right)$ for $j \in \mathbb{N}$,
and so we focus on the intermediate cases.
We establish that
$g(n,m) = (1+o(1)) \frac{m}{2}$ whp (with high probability)
when $n \ll m = n^{1+o(1)}$,
that $g(n,m) = (1+o(1)) \mu (\lambda) m$ whp
for a given function $\mu (\lambda)$
when $m \sim \lambda n$ for $\lambda > \frac{1}{2}$,
and that $g(n,m) = (1+o(1)) \frac{8s^{3}}{3n^{2}}$ whp
when $m = \frac{n}{2} + s$ for $n^{2/3} \ll s \ll n$.

We then also show that the genus of a fixed graph can increase dramatically if a small number of random edges are added. 
Given any connected graph with bounded maximum degree,
we find that the addition of $\epsilon n$ edges will whp result in a graph with genus $\Omega (n)$,
even when $\epsilon$ is an arbitrarily small constant!
We thus call this the `fragile genus' property. 
\end{abstract}
\maketitle

\section{Introduction}

\subsection{Background and motivation}

The Erd\H{o}s-R\'enyi random graph $G(n,m)$
(taken uniformly at random from the set of all labelled graphs with vertex set $[n] = \{ 1, 2, \ldots, n \}$ and exactly $m$ edges)
and the binomial random graph $G_{n,p}$
(the graph on $[n]$ where every edge occurs independently at random with probability $p$)
have been a source of fascination for many decades,
producing numerous exciting results
(see, for example, \cite{bol,fri,janluc} for research monographs devoted entirely to random graphs).

In this work,
we are interested in the genus of a graph.
A graph is said to have genus $g$ if this is the minimum number of handles that must be attached to a sphere in order to be able to embed the graph without any crossing edges.
Hence, the simplest case when $g=0$ corresponds to planar graphs.

The genus is one of the most fundamental properties of a graph,
and plays an important role in a number of applications and algorithms
(e.g.~colouring problems~\cite{moh} and the manufacture of electrical circuits~\cite{gib,lip}).
It is naturally intriguing to consider the genus of a random graph,
and such matters are also related to random graphs on surfaces
(see, for example,
Question 8.13 of~\cite{kmsfull}
and Section 9 of~\cite{dks}).
In addition,
results on the genus of random bipartite graphs~\cite{jing2}
were recently used to provide a polynomial-time approximation scheme for the genus of dense graphs~\cite{jing1}.

The genus of the binomial random graph was first studied by Archdeacon and Grable~\cite{arch},
who showed that $G_{n,p}$ has genus $(1+o(1)) \frac{pn^{2}}{12}$ with high probability
(whp for short, meaning with probability tending to $1$ as $n \to \infty$
--- see Definition~\ref{orderdefn})
if $p^{2}(1-p^{2}) \geq \frac{8(\log n)^{4}}{n}$.
A particularly notable consequence of this result (by taking $p = \frac{1}{2}$)
is that the classical uniform random graph $G(n)$
(taken uniformly at random from the set of all labelled graphs on $[n]$)
must then have genus $(1+o(1)) \frac{n^{2}}{24}$ whp.

As noted in~\cite{arch},
results for the genus of $G_{n,p}$ can be transferred into analogous results for the genus $g(n,m)$ of $G(n,m)$.
Taking into account later work by R\"odl and Thomas~\cite{rod}
(which deals with a substantially wider range for $p$),
these show that $g(n,m) = (1+o(1)) \frac{m}{6}$ whp when $m = \Theta (n^{2})$ and that $g(n,m) = (1+o(1)) \frac{jm}{2(j+2)}$ whp when $n^{1+ \frac{1}{j+1}} \ll m \ll n^{1+ \frac{1}{j}}$ for $j \in \mathbb{N}$.

Separately, important work has also been carried out to determine the probability that $G(n,m)$ is planar (i.e.~has zero genus)
when $m$ is comparatively small.
In particular, it is now well-known that $G(n,m)$ is planar whp when $m < \frac{n}{2} - \omega \left( n^{2/3} \right)$ (see~\cite{lucpitwie})
and that $\liminf \mathbb{P} [G(n,m) \textrm{ is planar}] > 0$
when $m = \frac{n}{2} + O \left( n^{2/3} \right)$ (see~\cite{lucpitwie} and~\cite{noy}).
For other interesting results in this area, 
see also~\cite{janpaper} and~\cite{kangluc}.

It is our aim here to bridge the gap between the $m \gg n^{1+ \frac{1}{j+1}}$ and $m = \frac{n}{2} + O \left( n^{2/3} \right)$ results.
We provide a thorough description of this intermediate region,
finding that there is different behaviour depending on whether
(i)~$n \ll m = n^{1+o(1)}$,
(ii)~$m \sim \lambda n$ for $\lambda > \frac{1}{2}$,
or (iii)~$m = \frac{n}{2} + s$ for $s>0$ satisfying $n^{2/3} \ll s \ll n$.

We then turn our attention to an interesting related problem,
concerning the genus of a graph that is partially random.
Here, we take a base graph,
and examine the supergraph formed by adding some random edges
(so $G(n,m)$ corresponds to the special case when the base graph is empty).

This type of model is sometimes called a `randomly perturbed' graph,
and was first introduced in~\cite{boh},
where the number of random edges needed for Hamiltonicity was studied
(the model is also related to the study of `smoothed analysis' of algorithms,
initiated in~\cite{spi}).

Subsequent work has then involved investigations of the clique number, chromatic number, diameter, and vertex-connectivity~\cite{bfkm};
subgraphs and Ramsey properties~\cite{krivsudtet};
expansion properties~\cite{flax, krivsmooth};
and subtrees~\cite{bot, krivtree};
as well as generalisations to hypergraphs and digraphs~\cite{flaxfri, krivdigraph, mcmy, sud}.

In this paper,
our focus is on the genus.
We take an arbitrary connected base graph $H$ with bounded maximum degree,
and examine the supergraph $G$ formed by adding random edges.
Rather surprisingly, we find that $G$ will whp have high genus,
even if $H$ has low genus and the number of random edges added is relatively small.
We thus call this the `fragile genus' property.

\subsection{Main results}

The main contributions of this paper are two-fold.
Firstly, we obtain a complete picture of the genus $g(n,m)$ 
of the Erd\H{o}s-R\'enyi random graph $G(n,m)$
for all values of $m$,
by producing precise results for the previously uncharted regions.
Secondly, we then initiate the study of how the genus of a fixed graph is affected when random edges are added,
discovering the fragile genus property.

Let us now present our main results in detail.
In the first of these,
we consider $g(n,m)$ for the region when $n \ll m = n^{1+o(1)}$
(e.g.~this would be the case for a function such as $m = n \ln n$).
Note that this is not an area that is covered by existing work,
but we obtain the following tight bounds:

\begin{Theorem} \label{deltacor}
Let $m=m(n)$ satisfy $n \ll m = n^{1+o(1)}$.
Then with high probability 
\begin{displaymath}
(1-o(1)) \frac{m}{2} \leq g(n,m) \leq \frac{m}{2}.
\end{displaymath}
\end{Theorem}

Perhaps the most obvious gap in previous knowledge concerns the case when $m$ is linear in $n$,
but above the threshold for planarity
(i.e.~the strictly supercritical regime).
We show that the genus behaves smoothly in this region:

\begin{Theorem} \label{lambdathm}
Let $m=m(n) \sim \lambda n$ for some fixed $\lambda > \frac{1}{2}$.
Then with high probability
\begin{displaymath}
g(n,m) = (1+o(1)) \mu (\lambda) m,
\end{displaymath}
where the function
$\mu: \left( \frac{1}{2}, \infty \right) \to \mathbb{R},
\lambda \mapsto \mu (\lambda)$
defined by

\begin{eqnarray*}
\mu (\lambda) = \frac{1}{4 \lambda^{2}}
\sum_{r=1}^{\infty} \frac{r^{r-2}}{r!}
\left( 2 \lambda e^{-2 \lambda} \right)^{r}
+ \frac{1}{2} \left(1 - \frac{1}{\lambda}\right)
\end{eqnarray*}
is strictly positive, monotonically increasing, continuous, and satisfies
$\mu (\lambda) \to 0$ as $\lambda \to \frac{1}{2}$
and $\mu (\lambda) \to \frac{1}{2}$ as $\lambda \to \infty$.
\end{Theorem}

One of the most fascinating areas of study in random graphs has been the behaviour of $G(n,m)$ when $m$ is close to $\frac{n}{2}$,
as many important features have been found to emerge around this key point.
Here, we examine in detail the slightly supercritical regime when $m = \frac{n}{2} + s$ for $s>0$ satisfying $n^{2/3} \ll s \ll n$
(i.e.~precisely the region between the planarity threshold and the linear case dealt with in Theorem~\ref{lambdathm}),
showing exactly how the genus grows:

\begin{Theorem} \label{sthm}
Let $m=m(n) = \frac{n}{2} + s(n)$,
where $s=s(n)$ satisfies $s>0$ for all $n$
and $n^{2/3} \ll s \ll n$.
Then with high probability
\begin{displaymath}
g(n,m) = (1+o(1)) \frac{8s^{3}}{3n^{2}}.
\end{displaymath}
\end{Theorem}

All these results are summarised in Table~\ref{table},
which gives an exciting picture of how the genus $g=g(n,m)$ behaves as $m$ grows.
In particular,
it is intriguing to see that the ratio of $g$ to $m$ increases from $0$ to $\frac{1}{2}$ until $m$ becomes superlinear in $n$,
after which it then decreases from $\frac{1}{2}$ to $\frac{1}{6}$
(see Section~\ref{discussion} for a discussion of this).

\begin{table} [ht]
\centering
\caption{
A summary of the genus $g:=g(n,m)$ of the random graph $G(n,m)$.
} 
\label{table}
{\renewcommand{\arraystretch}{1.5}
\begin{tabular}{|l||l||l|}
\hline
$m=\Theta \left( n^{2} \right)$ & 
$g = (1+o(1)) \frac{m}{6}$ \ whp &
See \cite{rod} \\
\hline
$n^{1+\frac{1}{j+1}} \ll m \ll n^{1+\frac{1}{j}}$ &
$g = (1+o(1)) \frac{jm}{2(j+2)}$ \ whp &
See \cite{rod} \\
\hline
$m= \Theta \left( n^{1+\frac{1}{j}} \right)$ &
$(1+o(1)) \frac{(j-1)m}{2(j+1)}$ &
See \cite{rod} \\
&
$\leq g \leq (1+o(1)) \frac{jm}{2(j+2)}$ \ whp &
\\
\hline
$n \ll m = n^{1+o(1)}$ &
$(1-o(1)) \frac{m}{2} \leq g \leq \frac{m}{2}$ \ whp &
Theorem~\ref{deltacor} \\
\hline
$m \sim \lambda n$, $\lambda > \frac{1}{2}$ &
$g = (1+o(1)) \mu (\lambda) m$ \ whp, &
Theorem~\ref{lambdathm} \\
&
where $\mu (\lambda) \to 0$ as $\lambda \to \frac{1}{2}$ &
\\
&
and $\mu (\lambda) \to \frac{1}{2}$ as $\lambda \to \infty$ &
\\
\hline
$m = \frac{n}{2}+s$, &
$g = (1+o(1)) \frac{8s^{3}}{3n^{2}}$ \ whp &
Theorem~\ref{sthm} \\
$s>0$ and $n^{2/3} \ll s \ll n$ &
&
\\
\hline
$m-\frac{n}{2} \sim c n^{2/3}$ &
$\lim_{n \to \infty} \mathbb{P} [g=0] = r(c) \in (0,1)$, &
See \cite{lucpitwie} \\
&
where $r(c) \to 1$ as $c \to - \infty$ &
\\
&
and $r(c) \to 0$ as $c \to \infty$ &
\\
\hline
$m < \frac{n}{2} - \omega \left( n^{2/3} \right)$ &
$g=0$ \ whp &
See \cite{lucpitwie} \\
\hline
\end{tabular}}
\end{table}

Finally, we turn our attention to our last main result,
which concerns the fragile genus property.
Here, we take an arbitrary connected graph $H$ with bounded maximum degree,
and a random graph $R$ on the same vertex set,
and we consider the genus $g(G)$ of the graph $G = H \cup R$.
We make an interesting discovery,
finding that $g(G)$ will whp be rather large,
even if $H$ and $R$ are both planar:

\begin{Theorem} \label{fragile}
Let $\Delta$ be a fixed constant,
and let $H=H(n,\Delta)$ be a connected graph with $n$ vertices and maximum degree at most $\Delta$.
Let $k=k(n) \to \infty$ as $n \to \infty$, 
and let $R=R(n,k)$ be a random graph on $V(H)$ consisting of exactly $k$ edges chosen uniformly at random from $\binom{V(H)}{2}$.
Let $G = G(n,\Delta,k) = H \cup R$.
Then with high probability
\begin{displaymath}
g(G) =
\Theta \left( \max \left\{ g(H),k \right\} \right).
\end{displaymath}
\end{Theorem}

Note that 
for $\limsup_{n \to \infty} \frac{k}{n} < \frac{1}{2}$,
the restriction on the maximum degree
in Theorem~\ref{fragile} is essential,
since otherwise we could take $H$ to be a star on $n$ vertices
(observe that 
whp the random graph $R$ would consist only of trees and unicyclic components,
and would consequently be outerplanar,
and so the overall graph $G$ would then have genus zero).

\subsection{Techniques and outline of the paper}

Our proofs typically utilise Euler's formula.
Given a graph $G$,
this states that the genus $g(G)$ satisfies
\begin{align*}
g(G) = \frac{1}{2} (e(G)-|G|-f(G)+\kappa(G)+1),
\end{align*}
where 
$e(G)$ is the number of edges of $G$,
$|G|$ is the number of vertices of $G$,
$f(G)$ is the number of faces of $G$ when embedded on a surface of minimal genus
(i.e.~a sphere to which $g(G)$ handles have been attached),
and $\kappa(G)$ is the number of components of $G$.

Consequently, our results often involve establishing new bounds for 
$f(G(n,m))$,
the number of faces of $G(n,m)$
when embedded on a surface of minimal genus.
For instance,
this might be achieved by first bounding the number of short faces through
probabilistic calculations on the number of short cycles,
and then separately bounding the number of larger faces using the fact that the total sum of all face sizes must be $2e(G)$.

We note that a key ingredient here is a new result (Corollary~\ref{faceequiv}) 
relating $f(G(n,m))$ to $f\left(G_{n,p}\right)$ for $p=p(n) = \frac{m}{\binom{n}{2}}$,
thus allowing us to work with the $G_{n,p}$ model when this is more convenient.
It is hoped that this may also prove to be of use to future researchers in this area.

Unfortunately, the number of short cycles may in fact be a gross over-estimate for the number of small faces if there are actually many large faces that consist of a short cycle with large trees rooted on the cycle
(see Figure~\ref{treefig}).

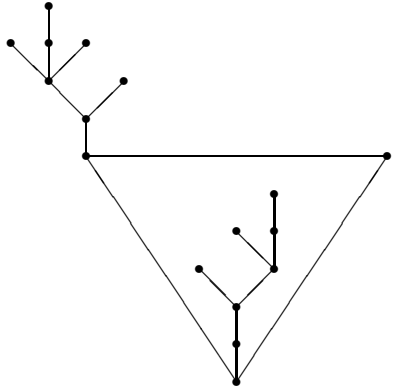
\begin{figure} [ht]
\setlength{\unitlength}{1cm}
\begin{picture}(5,5)(0,-0.5)
\put(1,2.5){\circle*{0.1}}
\put(5,2.5){\circle*{0.1}}
\put(3,-0.5){\circle*{0.1}}

\put(1,3){\circle*{0.1}}
\put(0.5,3.5){\circle*{0.1}}
\put(1.5,3.5){\circle*{0.1}}
\put(0,4){\circle*{0.1}}
\put(0.5,4){\circle*{0.1}}
\put(1,4){\circle*{0.1}}
\put(0.5,4.5){\circle*{0.1}}

\put(3,0){\circle*{0.1}}
\put(3,0.5){\circle*{0.1}}
\put(2.5,1){\circle*{0.1}}
\put(3.5,1){\circle*{0.1}}
\put(3.5,1.5){\circle*{0.1}}
\put(3,1.5){\circle*{0.1}}
\put(3.5,2){\circle*{0.1}}

\put(3,-0.5){\line(2,3){2}}
\put(3,-0.5){\line(-2,3){2}}
\put(1,2.5){\line(1,0){4}}

\put(1,2.5){\line(0,1){0.5}}
\put(0.5,3.5){\line(0,1){1}}
\put(1,3){\line(-1,1){1}}
\put(1,3){\line(1,1){0.5}}
\put(0.5,3.5){\line(1,1){0.5}}

\put(3,-0.5){\line(0,1){1}}
\put(3.5,1){\line(0,1){1}}
\put(3,0.5){\line(-1,1){0.5}}
\put(3,0.5){\line(1,1){0.5}}
\put(3.5,1){\line(-1,1){0.5}}
\end{picture}

\caption{An embedding with two large faces.} \label{treefig}
\end{figure}

Hence,
in order to attain the required level of accuracy,
we sometimes find it better to deal directly with the $2$-core of $G(n,m)$
(see Definition~\ref{2core})
rather than with the entire graph ---
note that this determines the overall genus.

Finally, the proof of Theorem~\ref{fragile} exploits a result from~\cite{krivnach} for decomposing the base graph $H$ into connected pieces of prescribed size.
We construct a particular minor of $G$ 
where each of these pieces is condensed into a vertex
(note that the genus of $G$ is at least the genus of any of its minors),
and we find that we can obtain our result by applying Theorem~\ref{lambdathm} to this minor.

We structure the paper as follows:
in Section~\ref{prelim},
we state the relevant terminology, notation, and key facts;
in Section~\ref{omega},
we begin our investigation of $g(n,m)$ with results for when $m = \omega (n)$,
proving Theorem~\ref{deltacor};
in Section~\ref{lambda},
we deal with the case when $m \sim \lambda n$ for $\lambda > \frac{1}{2}$,
proving Theorem~\ref{lambdathm};
in Section~\ref{s},
we fill in the remaining gap by determining the behaviour in the region $m = \frac{n}{2} + s$ for $n^{2/3} \ll s \ll n$,
proving Theorem~\ref{sthm};
in Section~\ref{contiguity},
we use our results to examine the contiguity
(see Definition~\ref{contiguitydefn})
of $G(n)$ and $G(n,m)$ with random graph models of given genus;
in Section~\ref{fragilesection},
we turn our attention to the fragile genus property,
proving Theorem~\ref{fragile};
and then finally,
in Section~\ref{discussion},
we discuss our results and the remaining open problems.

\section{Preliminaries} \label{prelim}

In this section,
we shall firstly (in Subsection~\ref{notation})
provide details of the notation and definitions that will be used throughout the paper,
and then (in Subsection~\ref{facts})
we shall present three important results that will be of great use to us.

\subsection{Notation and definitions} \label{notation}

Let us first note that
we shall always take $n$ and $m=m(n)$
to be integers satisfying $n>0$
and $m \geq 0$,
even if this is not always explicitly stated.

We start with the definitions of the standard random graph models:

\begin{Definition}
We shall let $G(n,m)$
denote a graph taken uniformly at random from
the set of all labelled graphs 
on the vertex set $[n] := \{1,2, \ldots, n\}$
with exactly $m=m(n)$ edges.

We shall let $G_{n,p}$ denote a graph on $[n]$
where every edge occurs independently at random with probability $p=p(n)$,
and we shall use $G(n)$ to denote $G_{n, \frac{1}{2}}$
(i.e.~a graph taken uniformly at random from the set of all labelled graphs on $[n]$).
\end{Definition}

Next, we state the notation to be used for various key characteristics:

\begin{Definition}
Given a graph $G$,
we shall use 
$|G|$ to denote the number of vertices of $G$,
$e(G)$ to denote the number of edges of $G$,
$g(G)$ to denote the genus of $G$,
$\kappa(G)$ to denote the number of components of $G$,
and $f(G)$ to denote the number of faces of $G$
when embedded on a surface of genus $g(G)$.

We also define random variables $g(n,m) := g(G(n,m))$, $\kappa(n,m) := \kappa(G(n,m))$, and $f(n,m) := f(G(n,m))$.

Given a particular embedding of a graph,
we shall use the \emph{length} of a face 
to mean the number of edges with a side in the face,
counting an edge twice if both sides are in the face
(for example,
the embedding shown in Figure~\ref{fig} has one face of length six and one face of length four).

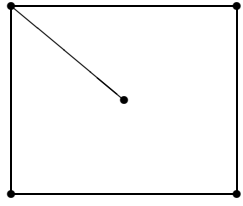
\begin{figure} [ht]
\setlength{\unitlength}{1cm}
\begin{picture}(3,2.5)(0,0)
\put(0,0){\line(1,0){3}}
\put(0,2.5){\line(0,-1){2.5}}
\put(3,2.5){\line(0,-1){2.5}}
\put(0,2.5){\line(1,0){3}}
\put(0,2.5){\line(6,-5){1.5}}
\put(0,0){\circle*{0.1}}
\put(3,0){\circle*{0.1}}
\put(1.5,1.25){\circle*{0.1}}
\put(0,2.5){\circle*{0.1}}
\put(3,2.5){\circle*{0.1}}
\end{picture}

\caption{An embedding with faces of length six and four.} \label{fig}
\end{figure}
\end{Definition}

We now also provide details of our order notation:

\begin{Definition} \label{orderdefn}
Given non-negative functions $a(n)$ and $b(n)$, 
we shall use the following notation:
\begin{itemize}
\item
$a(n) = \Omega (b(n))$ 
means there exists a constant $c>0$
such that $a(n) \geq c b(n)$ for all large $n$;
\item $a(n) = O(b(n))$ 
means there exists a constant $C$
such that $a(n) \leq C b(n)$ for all large $n$;
\item $a(n) = \Theta(b(n))$ 
means $a(n) = \Omega (b(n))$ and $a(n) = O(b(n))$;
\item $a(n) = \omega (b(n))$ or $a(n) \gg b(n)$ 
means $\frac{a(n)}{b(n)} \to \infty$ as $n \to \infty$;
\item $a(n) = o(b(n))$ or $a(n) \ll b(n)$ 
means $\frac{a(n)}{b(n)} \to 0$ as $n \to \infty$;
\item $a(n) \sim b(n)$ 
means $a(n) = (1+o(1))b(n)$.
\end{itemize}

We shall say that a random event $X_{n}$ happens
\emph{with high probability (whp)}
if
$\mathbb{P}(X_{n}) \to 1$ as $n \to \infty$. 
Given a non-negative random variable $a(n)$ 
and a non-negative function $b(n)$,
we shall use the following notation:
\begin{itemize}
\item $a(n) = \Omega (b(n))$ whp 
means there exists a constant $c>0$
such that $a(n) \geq c b(n)$ whp;
\item $a(n) = O(b(n))$ whp 
means there exists a constant $C$
such that $a(n) \leq C b(n)$ whp;
\item $a(n) = \Theta(b(n))$ whp 
means $a(n) = \Omega (b(n))$ whp and $a(n) = O(b(n))$ whp;
\item $a(n) = \omega (b(n))$ whp or $a(n) \gg b(n)$ whp 
means that,
given any constant $K$,
we have
$\frac{a(n)}{b(n)} > K$ whp;
\item $a(n) = o(b(n))$ whp or $a(n) \ll b(n)$ whp 
means that,
given any constant $\epsilon > 0$,
we have
$\frac{a(n)}{b(n)} < \epsilon$ whp;
\item $a(n) \sim b(n)$ whp 
means $a(n) = (1+o(1))b(n)$ whp.
\end{itemize}
\end{Definition}

Note that we shall always take all asymptotics to be as $n \to \infty$,
even if this is not always explicitly stated.

\subsection{Key facts} \label{facts}

In this subsection,
we shall formally present two important well-known results,
together with a new corollary.
The first is Euler's formula,
which we have already seen:

\begin{Theorem} [Euler's formula]
\label{euler}
Let $G$ be a graph.
Then
\begin{displaymath}
g(G) = \frac{1}{2} (e(G) - |G| - f(G) + \kappa(G) + 1).
\end{displaymath}
\end{Theorem}

For the second key result,
we first require the following definition:

\begin{Definition}
We say that a property is \emph{monotone increasing}
if whenever an edge is added to a graph with the property,
then the resulting graph also has the property.

Similarly, we say that a property is \emph{monotone decreasing}
if whenever an edge is deleted from a graph with the property,
then the resulting graph also has the property.

We say that a property is \emph{monotone} if it is either monotone increasing or monotone decreasing.
\end{Definition}

We may now state the aforementioned second crucial result:

\begin{Theorem} [see, for example, Proposition 1.15 of~\cite{janluc}] \label{equiv}
Given $m = m(n)$,
let $p = p(n) = \frac{m}{\binom{n}{2}}$.
Then if a monotone property holds whp for $G_{n,p}$,
it also holds whp for $G(n,m)$.
\end{Theorem}

Note that, for any function $x=x(n)$,
the property that $g(G) \leq x$ is monotone,
as is the property that $g(G) \geq x$.

Unfortunately, the same cannot be said if we replace $g(G)$ with $f(G)$,
the number of faces of $G$ when embedded on a surface of minimal genus.
For instance, let $C_{5}^{+}$ denote the graph formed be adding one edge to $C_{5}$,
let $K_{5}^{-}$ denote the graph formed be removing one edge from $K_{5}$,
and note that we have
$f\left(C_{5}\right)=2$, $f\left(C_{5}^{+}\right)=3$, $f\left(K_{5}^{-}\right)=6$, and $f\left(K_{5}\right)=5$
(observe that the first three graphs are planar, while $K_{5}$ has genus one).

Hence, adding an edge can actually increase or decrease (or have no impact on)
$f(G)$
(to be precise,
adding an edge between two components will leave $f(G)$ unchanged,
and adding an edge within a component will result in $f(G)$ either increasing by $1$ or decreasing by $1$,
depending on whether the genus stays the same or increases).

However,
the function $f(G)-e(G)$ is certainly monotone decreasing
(one way to see this is to note that Euler's formula gives
$f(G)-e(G) = \kappa(G)+1-|G|-2g(G)$,
and $\kappa(G)$ and $g(G)$ are clearly monotone decreasing and monotone increasing, respectively).
Using this, we may in fact still apply Theorem~\ref{equiv} to derive a useful new equivalence result for the number of faces:

\begin{Corollary} \label{faceequiv}
Let $m=m(n) \to \infty$ as $n \to \infty$,
let $p=p(n)=\frac{m}{\binom{n}{2}}$,
and suppose $x=x(n)$ is a function such that
$f\left(G_{n,p}\right) \leq x$ whp.
Then
\begin{displaymath}
f(n,m) \leq x + o(m) \quad \textrm{whp.}
\end{displaymath}
\end{Corollary}
\begin{proof}
We are required to show that,
given any constant $\epsilon > 0$,
we have $f(n,m) < x + \epsilon m$ whp.

Note that $e\left(G_{n,p}\right)$ has variance $\binom{n}{2}p(1-p) \leq m$,
and hence has standard deviation at most $m^{1/2}$,
which is $o(m)$ since $m \to \infty$.
Thus, since $e\left(G_{n,p}\right)$ has expectation exactly $\binom{n}{2}p=m$,
it follows that,
given any constant $\epsilon > 0$,
we have 
$e\left(G_{n,p}\right) > (1-\epsilon)m$ whp.
Therefore, 
since $f\left(G_{n,p}\right) \leq x$ whp,
we then have 
$f\left(G_{n,p}\right) - e\left(G_{n,p}\right) < x-(1-\epsilon)m$ whp.

Now recall our observation that $f(G)-e(G)$ is a monotone decreasing function,
from which it follows that the property that a graph satisfies 
$f(G)-e(G) < x-(1-\epsilon)m$ is monotone increasing.
Hence, we may apply Theorem~\ref{equiv},
thus obtaining $f(n,m)-m < x-(1-\epsilon)m$ whp,
i.e.~$f(n,m) < x + \epsilon m$ whp,
as desired.
\end{proof}

We note in passing that the $o(m)$ term in the statement of Corollary~\ref{faceequiv} can actually be reduced to $o \left( m^{\frac{1}{2}+\delta} \right)$ for any $\delta > 0$
(by the same proof).
However, we will not require such accuracy in this paper.

\section{$m=\omega(n)$: proof of Theorem~\ref{deltacor}} \label{omega}

We now come to our first main section,
where our focus is to produce a full account of $g(n,m)$ for the case $m \gg n$.
We shall start by stating (in Theorem~\ref{rodthm} and Theorem~\ref{rodcor})
the previously known results for this region,
before then completing the picture with a proof of Theorem~\ref{deltacor}.

The proof will employ Euler's formula,
and will hence involve obtaining results on the number of faces of $G(n,m)$.
The main work here will be done in Lemma~\ref{facelemma},
where we shall utilise Corollary~\ref{faceequiv} 
to allow us to work with the analogous $G_{n,p}$ model,
and then use probabilistic arguments
to bound the number of short cycles,
and hence the number of short faces,
and then the total number of faces.

Let us begin with the aforementioned two existing results from~\cite{rod},
which have been 
rephrased for $G(n,m)$ by applying Theorem~\ref{equiv}:

\begin{Theorem} [rephrased from Theorem 1.2 of \cite{rod}] \label{rodthm}
Let $m=m(n)$ satisfy 
$n^{1 + \frac{1}{j+1}} \ll m \ll n^{1 + \frac{1}{j}}$
for some fixed $j \in \mathbb{N}$.
Then
\begin{displaymath}
g(n,m) = (1+o(1)) \frac{jm}{2(j+2)} \quad \textrm{whp.}
\end{displaymath}
\end{Theorem}

\begin{Theorem} [rephrased from Theorem 1.4 of \cite{rod}] \label{rodcor}
(i) Let $m=m(n) = \Theta \left(n^{1 + \frac{1}{j}}\right)$ for some fixed $j \in \mathbb{N}_{>1}$.
Then
\begin{displaymath}
(1+o(1)) \frac{(j-1)m}{2(j+1)} \leq g(n,m) \leq (1+o(1)) \frac{jm}{2(j+2)} \quad \textrm{whp.}
\end{displaymath}
(ii) Let $m=m(n) = \Theta \left(n^{2}\right)$.
Then
\begin{displaymath}
g(n,m) = (1+o(1)) \frac{m}{6} \quad \textrm{whp.} 
\end{displaymath} 
\end{Theorem}

Note that we are left with a gap for the region $n \ll m = n^{1+o(1)}$,
which we shall fill with Theorem~\ref{deltacor}.
As mentioned,
the proof will require us to first obtain bounds on $f(n,m)$:

\begin{Lemma} \label{facelemma}
Let $m=m(n)$ satisfy both $m \to \infty$ as $n \to \infty$ 
and $m \ll n^{1+ \frac{1}{j}}$ for some fixed $j \in \mathbb{N}$.
Then 
\begin{displaymath}
f(n,m) \leq (1+o(1)) \frac{2}{j+2} m \quad \textrm{whp.}
\end{displaymath}
\end{Lemma}
\begin{proof}
We will use the $G_{n,p}$ model with $p = \frac{m}{\binom{n}{2}}$,
and show that the number of faces
is at most 
$(1+o(1)) \frac{2}{j+2} \binom{n}{2} p$ whp
(we will then be done,
by an application of Corollary~\ref{faceequiv}).
Thus, we are required to show that,
given any constant $\epsilon > 0$,
the number of faces is at most
$(1+\epsilon) \frac{2}{j+2} \binom{n}{2}p$ whp.

We will follow a similar argument to that used in the proof of Theorem~\ref{rodthm} (see~\cite{rod}),
which involves showing that whp $G_{n,p}$ will have few short cycles,
and hence few small faces,
and hence few faces in total.

Note that the expected number of cycles in $G = G_{n,p}$ of length at most $j+1$ is
\begin{eqnarray*}
\sum_{i=3}^{j+1} \binom{n}{i} \frac{i!}{2i} p^{i}
& \leq & \sum_{i=3}^{j+1} \frac{n^{i}p^{i}}{2i} \\
& \leq & \sum_{i=3}^{j+1} (np)^{i} \\
& \leq & (j+1) \max \left\{ np, (np)^{j+1} \right\} \quad (\textrm{since either } np \leq 1 \textrm{ or } np \geq 1) \\
& = & O \left( \max \left\{ np, (np)^{j+1} \right\} \right) \\
& = & O \left( (np) \max \left\{ 1, (np)^{j} \right\} \right) \\
& = & o(n^{2}p) \quad \left(\textrm{since } 1 \ll n \textrm{ and }
np = \frac{nm}{\binom{n}{2}} \ll \frac{n^{2+\frac{1}{j}}}{n^{2}} = n^{\frac{1}{j}} \right).
\end{eqnarray*}
Thus, by Markov's inequality,
we can say that
whp $G$ has no more than $\frac{1}{2(j+2)} \epsilon \binom{n}{2} p$
cycles of length at most $j+1$.

Let us now consider an embedding of $G$.
Note that the statement of this lemma is certainly true if $G$ is acyclic
(since then there is only one face),
so we may assume that $G$ is not acyclic,
in which case every face of the embedding must contain a cycle.

Let $f^{\prime}$ denote the number of faces in this embedding with length at most $j+1$.
Then every such face must contain a cycle of length at most $j+1$,
and every such cycle can only be included in at most two faces.
Hence,
whp we have
\begin{equation} \label{fprime2}
f^{\prime} \leq \frac{1}{j+2} \epsilon \binom{n}{2} p.
\end{equation}

Now let $f$ denote the total number of faces in this embedding,
and observe that
\begin{displaymath}
2 e(G) \geq 3 f^{\prime} + (j+2) (f - f^{\prime}) 
= (j+2)f - (j-1)f^{\prime}.
\end{displaymath}
Thus, we have
\begin{eqnarray*}
f & \leq & \frac{2}{j+2} e(G) + \frac{j-1}{j+2} f^{\prime} \\
& \leq & \frac{2}{j+2} e(G) + f^{\prime} \\
& \leq & \frac{2}{j+2} e(G) + \frac{1}{j+2} \epsilon \binom{n}{2} p
\quad
\textrm{whp \quad (by~\eqref{fprime2})} \\
& \leq & \frac{2}{j+2} \left( 1 + \frac{\epsilon}{2} \right) \binom{n}{2} p + \frac{2}{j+2} \frac{\epsilon}{2} \binom{n}{2} p
\quad
\textrm{whp} \\
& = & (1 + \epsilon) \frac{2}{j+2} \binom{n}{2} p,
\end{eqnarray*}
and so we are done.
\end{proof}

As a consequence of Lemma~\ref{facelemma},
we may derive the following important corollary:

\begin{Corollary} \label{facecor}
Let $m=m(n)$ satisfy both $m \to \infty$ as $n \to \infty$ and $m \leq n^{1+o(1)}$.
Then
\begin{displaymath}
f(n,m) = o(m) \quad \textrm{whp.}
\end{displaymath}
\end{Corollary}
\begin{proof}
We are required to show that,
given any constant $\epsilon > 0$,
we have $f(n,m) < \epsilon m$ whp.

We may simply choose a value $j \in \mathbb{N}$ such that
$j > \frac{2(1+\epsilon)}{\epsilon} - 2$,
in which case $\frac{2}{j+2} < \frac{\epsilon}{1+\epsilon}$.
Then, by Lemma~\ref{facelemma},
we have
\begin{eqnarray*}
f(n,m) & < & (1+\epsilon) \frac{2}{j+2} m 
\quad \textrm{whp} \\
& < & (1+\epsilon) \frac{\epsilon}{1+\epsilon} m \\
& = & \epsilon m,
\end{eqnarray*}
and so we are done.
\end{proof}

The proof of Theorem~\ref{deltacor} is now straightforward:

\begin{proof}[Proof of Theorem~\ref{deltacor}]
The upper bound holds for all $m$
--- we simply use Euler's formula
\begin{displaymath}
g(n,m) = \frac{1}{2} (m-n-f(n,m)+\kappa(n,m)+1)
\end{displaymath}
from Theorem~\ref{euler},
and observe that $n \geq \kappa(n,m)$ and $f(n,m) \geq 1$.

The lower bound also follows from Euler's formula,
using $n=o(m)$ and $f=o(m)$ whp by Corollary~\ref{facecor}.
\end{proof}

\section{$m \sim \lambda n \textrm{ for } \lambda > \frac{1}{2}$: proof of Theorem~\ref{lambdathm}} \label{lambda}

In this section,
we shall deal with the case when $m \sim \lambda n$ for $\lambda > \frac{1}{2}$ by proving Theorem~\ref{lambdathm},
showing that $g(n,m) = (1+o(1)) \mu (\lambda)m$
for the given function $\mu (\lambda)$
satisfying the various stated properties.

The proof again utilises Euler's formula and Corollary~\ref{facecor} on $f(n,m)$.
For Theorem~\ref{deltacor},
the role of $\kappa(n,m)$ was insignificant,
since we had $m \gg n \geq \kappa(n,m)$.
However, since we now have $m = O(n)$,
this time we find that we require more precise knowledge of $\kappa(n,m)$,
and we extract this (in Corollary~\ref{bol612}) from work in~\cite{bol}.

In order to establish the desired properties of $\mu (\lambda)$,
we shall find it helpful to first consider a related function $u$
(see Defintion~\ref{udefn}),
and so the first half of this section will involve investigating the latter.
We start with the definition:

\begin{Definition} \label{udefn}
Let the function
$u:[0,\infty) \to \mathbb{R}, c \mapsto u(c)$ be defined by
\begin{displaymath}
u(c) = \frac{1}{c} \sum_{r=1}^{\infty} \frac{r^{r-2}}{r!} \left( ce^{-c} \right)^{r}.
\end{displaymath}
\end{Definition}

Note that this function is well-defined,
since the function $c \mapsto ce^{-c}$ is maximised at $c=1$
and so
$\frac{r^{r-2}}{r!} \left( ce^{-c} \right)^{r} \leq \frac{r^{r-2}}{r!} e^{-r} 
\leq \frac{r^{-5/2}}{\sqrt{2 \pi}}$ by Stirling's bound,
which means that the sum does indeed converge
(and at $c=0$, we have $u(c)=1$).

Let us next observe two fundamental properties of the function $u$:

\begin{Lemma} [\cite{bol}, remark following Theorem 5.12] \label{bolremark}
	\begin{displaymath}
		u(c) = 1 - \frac{c}{2} \textrm{ for } c \in [0,1].
	\end{displaymath}
\end{Lemma}

\begin{Lemma} \label{ucont}
	The function $u(c)$ is continuous for $c \ge 0$.
\end{Lemma}
\begin{proof}
	Since we already have continuity at $c=0$ by Lemma~\ref{bolremark},
	it now suffices for us to show that the sum
	$\sum_{r=1}^{\infty} \frac{r^{r-2}}{r!} \left( ce^{-c} \right)^{r}$
	is uniformly convergent.
	
	But since the function $c \mapsto ce^{-c}$ is maximised at $c=1$,
	this reduces to just showing that
	$\sum_{r=1}^{\infty} \frac{r^{r-2}}{r!} e^{-r}$
	converges.
	For this, we may then simply use Stirling's bound $r! \geq \sqrt{2 \pi r} \left( \frac{r}{e} \right)^{r}$
	and the observation that $\sum_{r=1}^{\infty} r^{-5/2}$ converges,
	and we are done.
\end{proof}

As mentioned,
the proof of Theorem~\ref{lambdathm} will use Euler's formula,
and will require us to collect accurate information on $\kappa (n,m)$.
The following two results relating this to the function $u$ will consequently be extremely useful:

\begin{Theorem} [\cite{bol}, Theorem 6.13] \label{bol613}
Let $m=m(n) = \frac{cn}{2}$ for $c=c(n) \le 8 \log n$,
and let $\beta \in \left( \frac{2}{3}, 1 \right)$ be a fixed constant.
Then
\begin{displaymath}
\mathbb{P} \left[ 
\big| \kappa (n, m) - u(c)n \big|
\geq (\log n)n^{\beta}
\right]
\leq 100 n^{1-2\beta} (\log n)^{-1}.
\end{displaymath}
\end{Theorem}

\begin{Corollary} \label{bol612}
Let $\lambda \geq 0$ be a fixed constant.
Then
\begin{displaymath}
\kappa (n, \lfloor \lambda n \rfloor) = (1+o(1))u(2\lambda)n \quad \textrm{whp.}
\end{displaymath}
\end{Corollary}
\begin{proof}
	By Theorem~\ref{bol613}
	(with $c=c(n) = \frac{2 \lfloor \lambda n \rfloor}{n}$),
	we have
	\begin{displaymath}
	\kappa (n, \lfloor \lambda n \rfloor) = (1+o(1))u \left( \frac{2 \lfloor \lambda n \rfloor}{n} \right) n \quad \textrm{whp,}
	\end{displaymath}
	and so the result then follows from the continuity of $u$
	(see Lemma~\ref{ucont}).
\end{proof}

We may also derive a second helpful corollary:

\begin{Corollary} \label{expcor}
Let $\lambda \geq 0$ be a fixed constant.
Then
\begin{displaymath}
u(2\lambda) = \lim_{n \to \infty} \frac{1}{n}
\mathbb{E} \left[ \kappa (n, \lfloor \lambda n \rfloor) \right].
\end{displaymath}
\end{Corollary}
\begin{proof}
Note first that Corollary~\ref{bol612} implies that
$u(2 \lambda) \le 1$,
and hence we must always have
\begin{equation} \label{newu1}
\left| \kappa (n, \lfloor \lambda n \rfloor) - u(2\lambda)n \right| \le n.
\end{equation}

Given any $\epsilon > 0$,
Corollary~\ref{bol612} also implies that there exists $N=N(\epsilon)$ such that
\begin{equation} \label{newu2}
\mathbb{P} [ \left| \kappa (n, \lfloor \lambda n \rfloor) - u(2\lambda)n \right| > \epsilon n] < \epsilon 
\end{equation}
for all $n \ge N$.

Hence,
by combining~\eqref{newu1} and~\eqref{newu2},
we obtain
\begin{displaymath}
\mathbb{E} [ \left| \kappa (n, \lfloor \lambda n \rfloor) - u(2\lambda)n \right| ] \leq 2 \epsilon n
\end{displaymath}
for all $n \ge N$.

Thus,
since $\epsilon$ was arbitrary,
we have
\begin{displaymath}
\mathbb{E} \left[ \kappa (n, \lfloor \lambda n \rfloor) \right] = u(2\lambda)n + o(n),
\end{displaymath}
and so we are done.
\end{proof}

We shall now prove one final lemma,
after which we shall then have all the ingredients ready for our proof of Theorem~\ref{lambdathm}:

\begin{Lemma} \label{derivlemma}
For all $c>0$,
the derivative $u^{\prime} (c)$ exists and is monotonically increasing
(and so the function $u(c)$ is convex for $c \geq 0$).
\end{Lemma}
\begin{proof}
We shall first look to establish that the derivative $u^{\prime} (c)$ exists for all $c>0$.

Let $\delta \in (0,1)$ be a constant,
and note (by Lemma~\ref{bolremark}) that it suffices for us to consider $c \geq \delta$
(it would not be enough just to consider $c \ge 1$,
since we need to rule out the possibility that the derivatives from the left and right 
at $c=1$ are different).
Hence, we must show that the sum of derivatives
\begin{displaymath}
\sum_{r=1}^{\infty} \left( \frac{d}{dc} \left( \frac{1}{c} \frac{r^{r-2}}{r!} \left( ce^{-c} \right)^{r} \right) \right) 
\end{displaymath}
is uniformly convergent in any compact interval of $[\delta, \infty)$.

We have
\begin{eqnarray*}
\frac{d}{dc} \left( \frac{1}{c} \frac{r^{r-2}}{r!} \left( ce^{-c} \right)^{r} \right) & = &
\frac{1}{c} \frac{r^{r-2}}{r!} \left( ce^{-c} \right)^{r-1} re^{-c} (1-c) 
- \frac{1}{c^{2}} \frac{r^{r-2}}{r!} \left( ce^{-c} \right)^{r} \\
& = & \frac{(1-c)}{c^{2}} \frac{r^{r-1}}{r!} \left( ce^{-c} \right)^{r}
- \frac{1}{c^{2}} \frac{r^{r-2}}{r!} \left( ce^{-c} \right)^{r}.
\end{eqnarray*}
Thus, since $\delta > 0$,
it suffices for us to show that the sums
$\sum_{r=1}^{\infty} \frac{r^{r-1}}{r!} \left( ce^{-c} \right)^{r}$
and
$\sum_{r=1}^{\infty} \frac{r^{r-2}}{r!} \left( ce^{-c} \right)^{r}$
are both uniformly convergent.

But recall that we have already shown the latter during the proof of Lemma~\ref{ucont},
and note that the former follows from the same argument
(using the convergence of
$\sum_{r=1}^{\infty} r^{-3/2}$
instead of $\sum_{r=1}^{\infty} r^{-5/2}$).

Hence,
the derivative $u^{\prime} (c)$ exists for all $c>0$. \\

We shall now show that $u^{\prime} (c)$ is monotonically increasing for all $c>0$.
Note that it is difficult to do this by investigating the second derivative $u^{\prime\prime}(c)$,
since adapting the uniform convergence arguments above would this time lead us to a comparison with
$\sum_{r=1}^{\infty} r^{-1/2}$,
which is divergent.
Hence,
we shall instead introduce an alternative method involving Corollary~\ref{expcor}.

Observe that the graph $G(n,m)$ may be constructed by adding edges one-by-one uniformly at random
(thus inducing an ordering of these edges).
Hence, if we let
\begin{displaymath}
p(n,i) := \mathbb{P}[\textrm{adding edge $i$ will reduce the number of components}],
\end{displaymath}
then for $m_{1} \leq m_{2}$ we have
\begin{displaymath}
\mathbb{E} \left[ \kappa (n,m_{1}) - \kappa (n,m_{2}) \right] = \sum_{i=m_{1}+1}^{m_{2}} p(n,i).
\end{displaymath}

Since we have already seen that $u^{\prime}(c)$ exists,
we have
\begin{eqnarray*}
u^{\prime}(c) & = & \lim_{\epsilon \to 0} \frac{u(c+\epsilon)-u(c)}{\epsilon} \\
& = & \lim_{\epsilon \to 0^+} \frac{u(c+\epsilon)-u(c)}{\epsilon} \\
& = & \lim_{\epsilon \to 0^+} \frac{1}{\epsilon} \lim_{n \to \infty} \frac{1}{n} \mathbb{E} \left[ \kappa \left( n, \left\lfloor \frac{(c+\epsilon)n}{2} \right\rfloor \right) - \kappa \left( n, \left\lfloor \frac{cn}{2} \right\rfloor \right) \right] \textrm{ \quad (by Corollary~\ref{expcor})} \\
& = & - \lim_{\epsilon \to 0^+} \frac{1}{\epsilon} \lim_{n \to \infty} \frac{1}{n} \sum_{i = \left\lfloor \frac{cn}{2} \right\rfloor +1}^{\left\lfloor \frac{(c+\epsilon)n}{2} \right\rfloor} p(n,i).
\end{eqnarray*}

Observe that, for each $n$, the function $p(n,i)$ is monotonically decreasing in $i$.
Thus, given $c_{1}$ and $c_{2}$ with $c_{1} \leq c_{2}$,
we have
\begin{displaymath}
\sum_{i = \left\lfloor \frac{c_{2}n}{2} \right\rfloor +1}^{\left\lfloor \frac{\left( c_{2}+\epsilon \right) n}{2} \right\rfloor} p(n,i)
\leq 
1+
\sum_{i = \left\lfloor \frac{c_{1}n}{2} \right\rfloor +1}^{\left\lfloor \frac{ \left( c_{1}+\epsilon \right) n}{2} \right\rfloor} p(n,i)
\end{displaymath}
(where the `$1+$' term comes from taking into account the possibility that the number of terms in the first sum could be one greater than the number of terms in the second sum).

Hence,
\begin{displaymath}
\lim_{n \to \infty} \frac{1}{n}
\sum_{i = \left\lfloor \frac{c_{2}n}{2} \right\rfloor +1}^{\left\lfloor \frac{\left( c_{2}+\epsilon \right) n}{2} \right\rfloor} p(n,i)
\leq 
\lim_{n \to \infty} \frac{1}{n}
\sum_{i = \left\lfloor \frac{c_{1}n}{2} \right\rfloor +1}^{\left\lfloor \frac{ \left( c_{1}+\epsilon \right) n}{2} \right\rfloor} p(n,i),
\end{displaymath}
and so $u^{\prime}(c)$ is indeed monotonically increasing.
\end{proof}

We now conclude this section with the proof of our main result:

\begin{proof}[Proof of Theorem~\ref{lambdathm}]
To prove the result,
including showing that the function $\mu (\lambda)$ satisfies the various properties stated in the theorem,
we shall show 
\begin{itemize}
\item[(i)] $\mu (\lambda)$ is continuous for $\lambda \geq 0$; 
\item[(ii)] $\mu \left( \frac{1}{2} \right) = 0$
(and hence $\mu (\lambda) \to 0$ as $\lambda \to \frac{1}{2}$ by (i)); 
\item[(iii)] $\mu(\lambda) \to \frac{1}{2}$ as $\lambda \to \infty$; 
\item[(iv)] $g(n,m) = \mu (\lambda) m + o(m)$ whp for $\lambda > \frac{1}{2}$; 
\item[(v)] $\mu (\lambda)$ is strictly positive for $\lambda > \frac{1}{2}$; 
\item[(vi)] $\mu (\lambda)$ is monotonically increasing for $\lambda \geq 0$. 
\end{itemize}

(i) Proof that $\mu (\lambda)$ is continuous for $\lambda \geq 0$:

Note that
\begin{equation} \label{mueqn}
\mu (\lambda) = \frac{u(2\lambda) + \lambda - 1}{2\lambda}.
\end{equation}
Hence, Property~(i) follows immediately from the continuity of $u$
(see Lemma~\ref{ucont}). \\

(ii) Proof that $\mu \left( \frac{1}{2} \right) = 0$:

This is established by Lemma~\ref{bolremark} and~\eqref{mueqn}. \\

(iii) Proof that $\mu(\lambda) \to \frac{1}{2}$ as $\lambda \to \infty$:

This follows from the observations that $2 \lambda e^{-2 \lambda} \leq e^{-1}$ and that $\sum_{r=1}^{\infty} \frac{r^{r-2}}{r!} e^{-r}$ is convergent
(e.g.~by Stirling's bound,
as with Definition~\ref{udefn}). \\

(iv) Proof that $g(n,m) = \mu (\lambda) m + o(m)$ whp for $\lambda > \frac{1}{2}$:

Recall that we are considering $m \sim \lambda n$,
and recall also that Euler's formula (Theorem~\ref{euler}) gives
$g(n,m) = \frac{1}{2} (m-n-f(n,m) + \kappa(n,m) + 1)$.

By Corollary~\ref{bol612},
we have $\kappa (n, \lfloor \lambda n \rfloor) = (1+o(1))u(2\lambda)n$ whp.
Therefore, by the monotonicity of $\kappa$ 
(with respect to the number of edges)
and the continuity of $u$ (see Lemma~\ref{ucont}),
we also have
$\kappa (n, m) = (1+o(1))u(2\lambda)n$ whp.
Hence,
$\kappa (n,m) = (1+o(1)) (2 \mu(\lambda) + \frac{1}{\lambda} - 1)m$ whp.

Thus, noting from Corollary~\ref{facecor} that 
$f(n,m) = o(m)$ whp,
we are done. \\

(v) Proof that $\mu (\lambda)$ is strictly positive for $\lambda > \frac{1}{2}$:

It follows from Property (iv) that we must certainly have $\mu (\lambda) \geq 0$ for all $\lambda > \frac{1}{2}$.
Hence, let us suppose that there exists $\lambda > \frac{1}{2}$ for which $\mu (\lambda) = 0$,
and note that
(again using Property~(iv))
we then have $g \left( n, \lceil \lambda n \rceil \right) = o(m) = o(n)$.

Let $s = s(n) := \max \left \{ n^{3/4}, \left( \frac{n^{2}}{2} g \left( n, \lceil \lambda n \rceil \right) \right) ^{1/3} \right \} = o(n)$.
Then $s>0$ and $n^{2/3} \ll s \ll n$.

But (by monotonicity)
$g \left( n, \lceil \lambda n \rceil \right) \geq g \left( n, \left \lceil \frac{n}{2} + s \right \rceil \right)$
for large $n$,
since $\lambda > \frac{1}{2}$,
and (by Theorem~\ref{sthm}) whp
$g \left( n, \left \lceil \frac{n}{2} + s \right \rceil \right)
= (1+o(1)) \frac{8s^{3}}{3n^{2}} \geq (1+o(1)) \frac{4}{3} g(n, \lceil \lambda n \rceil)$,
by definition of $s$.
Thus, we obtain a contradiction.
\\

(vi) $\mu (\lambda)$ is monotonically increasing for $\lambda \geq 0$: 

We shall show $\mu^{\prime}(\lambda) \geq 0$ for $\lambda > 0$.
Since we know that $u$ is differentiable
(see Lemma~\ref{derivlemma}),
it follows from~\eqref{mueqn} that we have
\begin{eqnarray*}
\mu^{\prime} (\lambda) & = &
\frac{1}{4 \lambda ^{2}} \left( 2 \lambda \left( u^{\prime} (2 \lambda) + 1 \right) \right)
- 2 (u(2 \lambda) + \lambda - 1) \\
& = & \frac{1}{2 \lambda ^{2}} \left( \lambda u^{\prime} (2 \lambda) - u (2 \lambda) + 1 \right)
\end{eqnarray*}
(where $u^{\prime} (2 \lambda)$ denotes $\frac{d}{d \lambda} u (2 \lambda)$).
Hence, it suffices to show that
$\lambda u^{\prime} (2 \lambda) - u (2 \lambda) + 1 \geq 0$ for all $\lambda > 0$.

We shall accomplish this by establishing that 
\begin{itemize}
\item[(a)] $\lambda u^{\prime} (2 \lambda) - u (2 \lambda) + 1 = 0$ for all $\lambda \in \left( 0, \frac{1}{2} \right)$; 
\item[(b)] $\lambda u^{\prime} (2 \lambda) - u (2 \lambda)$ is monotonically increasing.
\end{itemize}

Note that (a) follows immediately from Lemma~\ref{bolremark}.

For (b), let $\epsilon \geq 0$ and observe that
\begin{eqnarray*}
& & (\lambda + \epsilon) u^{\prime} (2(\lambda + \epsilon))
- u(2(\lambda + \epsilon))
- \lambda u^{\prime} (2\lambda) 
+ u(2\lambda) \\
& = & \lambda \Big( u^{\prime} (2(\lambda + \epsilon)) - u^{\prime} (2\lambda) \Big)
- \epsilon \left( \frac{u(2(\lambda + \epsilon)) - u(2\lambda)}{\epsilon}
- u^{\prime} (2(\lambda + \epsilon)) \right).
\end{eqnarray*}
It follows from the convexity of $u$
(see Lemma~\ref{derivlemma}) that the second bracket is at most $0$,
and it follows from the monotonicity of $u^{\prime}$
(again, see Lemma~\ref{derivlemma})
that the first bracket is at least $0$,
and so we are done.
\end{proof}

\section{$m=\frac{n}{2}+s$: proof of Theorem~\ref{sthm}} \label{s}

In this section,
we shall prove Theorem~\ref{sthm},
which fills in the remaining gap in our picture of $g(n,m)$
by dealing with the case when $m = \frac{n}{2} + s$ for positive $s$ satisfying $n^{2/3} \ll s \ll n$.

The proof will again involve an application of Euler's formula.
However, in order to achieve the desired level of precision,
this time we shall actually work directly with the $2$-core
(see Definition~\ref{2core})
of the largest component of $G(n,m)$,
rather than with the entire graph.

We shall find it helpful to begin by first defining the following concepts,
which are taken 
(with some slight rewording)
from Section 2 of~\cite{luccyc}:

\begin{Definition} \label{leafneighbourdefn}
Given a cycle $C$ in a graph $G$, 
let us define the \emph{leaf neighbourhood} $T(C)$ of $C$
to consist of all trees rooted at $C$
(formally, $T(C)$ is the union of any tree components in $G \setminus V(C)$
that are attached to $C$ by exactly one edge).

Let us also split all neighbours of $C$ which do not belong to $T(C)$ into two groups ---
a vertex of $G \setminus (C \cup T(C))$ adjacent to exactly one vertex of $C$ will be called a \emph{good neighbour},
while a vertex of $G \setminus (C \cup T(C))$ adjacent to more than one vertex of $C$ will be called
a \emph{bad neighbour}.
\end{Definition}

An illustration of these definitions is given in Figure~\ref{leaffig}.
Here, vertices on the cycle $C$ are indicated by $c$,
vertices in the leaf neighbourhood $T(C)$ are indicated by $t$,
and the good and bad neighbours are indicated by $g$ and $b$, respectively.

\begin{figure} [ht]
\setlength{\unitlength}{1cm}
\begin{picture}(6.5,4.7)(0,-1)
\put(1,2.5){\circle*{0.1}}
\put(5,2.5){\circle*{0.1}}
\put(3,-0.5){\circle*{0.1}}
\put(0.85,2.2){$c$}
\put(4.9,2.2){$c$}
\put(2.9,-0.8){$c$}

\put(1,3){\circle*{0.1}}
\put(0.5,3.5){\circle*{0.1}}
\put(1.5,3.5){\circle*{0.1}}
\put(0.75,2.7){$t$}
\put(0.4,3.65){$t$}
\put(1.4,3.65){$t$}

\put(3,0){\circle*{0.1}}
\put(3,0.5){\circle*{0.1}}
\put(2.5,1){\circle*{0.1}}
\put(3.5,1){\circle*{0.1}}
\put(3.5,1.5){\circle*{0.1}}
\put(3,1.5){\circle*{0.1}}
\put(3.5,2){\circle*{0.1}}
\put(3.05,-0.1){$t$}
\put(2.7,0.3){$t$}
\put(2.55,1){$t$}
\put(3.55,1){$t$}
\put(3.55,1.5){$t$}
\put(3.05,1.5){$t$}
\put(3.55,2){$t$}

\put(3,-0.5){\line(2,3){2}}
\put(3,-0.5){\line(-2,3){2}}
\put(1,2.5){\line(1,0){4}}

\put(1,2.5){\line(0,1){0.5}}

\put(1,3){\line(-1,1){0.5}}
\put(1,3){\line(1,1){0.5}}

\put(3,-0.5){\line(0,1){1}}
\put(3.5,1){\line(0,1){1}}
\put(3,0.5){\line(-1,1){0.5}}
\put(3,0.5){\line(1,1){0.5}}
\put(3.5,1){\line(-1,1){0.5}}

\put(1.5,3){\circle*{0.1}}
\put(1,2.5){\line(1,1){0.5}}
\put(1.55,2.9){$t$}

\put(1,2.5){\line(-1,0){1}}
\put(0.5,2.5){\circle*{0.1}}
\put(0,2.5){\circle*{0.1}}
\put(0,2.5){\line(0,-1){0.5}}
\put(0.5,2.5){\line(-1,-1){0.5}}
\put(0,2){\circle*{0.1}}
\put(0.25,2.65){$g$}

\put(3,3){\circle*{0.1}}
\put(3,3){\line(1,-1){0.5}}
\put(3,3){\line(-1,-1){0.5}}
\put(2.5,2.5){\circle*{0.1}}
\put(3.5,2.5){\circle*{0.1}}
\put(2.9,3.1){$b$}
\put(2.35,2.6){$c$}
\put(3.5,2.6){$c$}

\put(5,2.5){\line(1,0){1}}
\put(5.5,2.5){\line(0,1){0.5}}
\put(6,2.5){\line(0,-1){1}}
\put(5.5,2.5){\line(1,-1){0.5}}
\put(6,1.5){\line(1,-1){0.5}}
\put(6,1.5){\line(-1,-1){0.5}}
\put(5.5,2.5){\circle*{0.1}}
\put(6,2.5){\circle*{0.1}}
\put(5.5,3){\circle*{0.1}}
\put(6,2){\circle*{0.1}}
\put(6,1.5){\circle*{0.1}}
\put(6.5,1){\circle*{0.1}}
\put(5.5,1){\circle*{0.1}}
\put(5.35,2.15){$g$}

\put(4,1){\circle*{0.1}}
\put(3.6,0.4){\circle*{0.1}}
\put(4,1){\line(1,-1){0.7}}
\put(3.6,0.4){\line(1,-1){0.7}}
\put(3.95,0.05){\line(2,3){0.4}}
\put(4.35,0.65){\circle*{0.1}}
\put(3.95,0.05){\circle*{0.1}}
\put(4.7,0.3){\circle*{0.1}}
\put(4.3,-0.3){\circle*{0.1}}
\put(4.4,0.65){$g$}
\put(3.8,-0.25){$g$}
\put(4.1,0.95){$c$}
\put(3.7,0.35){$c$}

\put(3,-0.5){\line(-1,0){0.5}}
\put(3,-0.5){\line(-1,-1){0.5}}
\put(2.5,-1){\line(0,1){0.5}}
\put(2.5,-0.5){\circle*{0.1}}
\put(2.5,-1){\circle*{0.1}}
\put(2.15,-0.5){$g$}
\put(2.15,-1){$g$}

\put(1.3,0.8){\line(1,0){0.8333}}
\put(1.3,0.8){\line(-1,1){0.5}}
\put(0.8,0.3){\line(1,1){1}}
\put(1.8,1.3){\circle*{0.1}}
\put(1.3,0.8){\circle*{0.1}}
\put(0.8,1.3){\circle*{0.1}}
\put(0.8,0.3){\circle*{0.1}}
\put(2.1333,0.8){\circle*{0.1}}
\put(1.85,1.3){$c$}
\put(1.25,0.5){$b$}
\put(2,0.5){$c$}

\put(5.5,0){\line(1,0){1}}
\put(6,0){\line(0,-1){0.5}}
\put(5.5,0){\circle*{0.1}}
\put(6,0){\circle*{0.1}}
\put(6.5,0){\circle*{0.1}}
\put(6,-0.5){\circle*{0.1}}
\end{picture}

\caption{An example of a cycle,
its good and bad neighbours,
and its leaf neighbourhood.} \label{leaffig}
\end{figure}
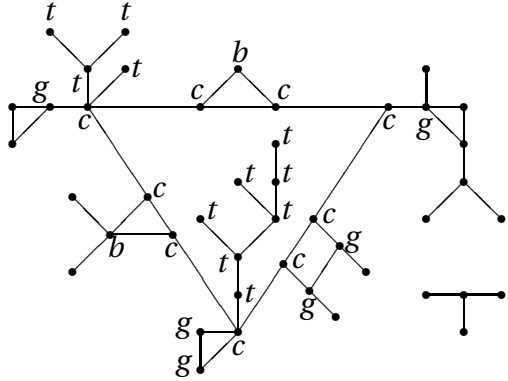

The following result,
which is derived from details within the proof of Theorem 3 of~\cite{luccyc}, 
will be extremely useful:

\begin{Lemma} \label{Poisson}
Let $m=m(n) = \frac{n}{2} + s(n)$,
where $s=s(n)$ satisfies $s>0$ for all $n$
and $n^{2/3} \ll s \ll n$.
Let $x = x(n) = 0.05 \log \left( \frac{s^{3}}{n^{2}} \right)$,
and let $Z(n,i)$ denote the number of cycles in $G(n,m)$ with 
\begin{itemize}
\item[(i)] length at most $\frac{in}{s}$; 
\item[(ii)] leaf neighbourhood at most $\frac{xn^{2}}{s^{2}}$; 
\item[(iii)] between $1$ and $\frac{xn}{s}$ good neighbours; 
\item[(iv)] no bad neighbours. 
\end{itemize}

Then for each fixed $i \in [0,\infty)$,
there exists a value $\lambda(i) \geq 0$ 
and a random variable 
$Y(i) \sim \textrm{Poi} (\lambda(i)) := \textrm{Poisson} \ (\lambda(i))$ 
such that
\begin{displaymath}
Z(n,i) = Y(i) \quad \textrm{whp,}
\end{displaymath}
where $\lambda: [0,\infty) \to [0,\infty), i \mapsto \lambda (i)$ 
is a monotonically increasing function satisfying
$\lambda (0) = 0$ and $\lambda (i) \to \infty$ as $i \to \infty$.
\footnote{We use $X \sim \textrm{Poi}(0)$ to mean $\mathbb{P}[X=0]=1$. \label{note}}
\end{Lemma}
\begin{proof}
Let 
\begin{displaymath}
\lambda (i) = \frac{1}{\sqrt{8 \pi}}
\int_{0}^{i} \int_{0}^{\infty}
(e^{4x}-1)y^{-1.5} \exp \left( - \frac{x^{2}}{2y} - 2y \right)
\mathrm{d}y \textrm{ } \mathrm{d}x.
\end{displaymath}
It is observed within the proof of Theorem 3 of \cite{luccyc} that
$\lambda (i) \to \infty$ as $i \to \infty$,
and it is shown that
for each $i \in (0,\infty)$
there exists $Y(i) \sim \textrm{Poi} (\lambda(i))$ such that
\begin{displaymath}
Z(n,i) \stackrel{\mathbb{P}}{\to} Y(i)
\end{displaymath}
(i.e.~for all $\delta > 0$,
we have
$\mathbb{P}[|Z(n,i) - Y(i)| > \delta] \to 0$ as $n \to \infty$).

Since $Z(n,i)$ and $Y(i)$ are always integer-valued,
it is then immediate
(e.g.~by considering $\delta = \frac{1}{2}$)
that we actually have $Z(n,i) = Y(i)$ whp.
The extension of this result to include the case $i=0$ is trivial
(albeit with a slight abuse of notation --- see Footnote~\ref{note}),
and the observation that $\lambda(i)$ is monotonically increasing just comes from the fact that the integrand is always non-negative.
\end{proof}

We shall also use two other facts from~\cite{luccyc}:

\begin{Lemma} [\cite{luccyc}, Fact 8] \label{Fact8}
Let $m=m(n) = \frac{n}{2} + s(n)$,
where $s=s(n)$ satisfies $s>0$ for all $n$
and $n^{2/3} \ll s \ll n$.
Then whp $G(n,m)$ contains no subgraphs of size less than 
$0.1 \frac{n}{s} \log \left( \frac{s^{3}}{n^{2}} \right)$
which have more edges than vertices.
\end{Lemma}

\begin{Lemma} [\cite{luccyc}, Fact 9] \label{Fact9}
Let $m=m(n) = \frac{n}{2} + s(n)$,
where $s=s(n)$ satisfies $s>0$ for all $n$
and $n^{2/3} \ll s \ll n$,
and let $a=a(n)$ satisfy $a \to \infty$ as $n \to \infty$ 
but $a < \log \left( \frac{s^{3}}{n^{2}} \right)$.
Then whp every cycle in $G(n,m)$ of length less than $\frac{an}{s}$
has a leaf neighbourhood smaller than $\frac{a^{2}n^{2}}{s^{2}}$
and less than $\frac{a^{2}n}{s}$ neighbours.
\end{Lemma}

As mentioned,
one further ingredient will be the concept of the $2$-core:

\begin{Definition} \label{2core}
Let us define the \emph{$2$-core} of a graph to be the subgraph formed by repeatedly deleting all vertices with degree less than two.
Equivalently,
the $2$-core of a graph is the maximal subgraph with minimum degree at least two.
\end{Definition}

Observe that deleting a vertex with degree zero or one cannot change the genus of a graph,
and so it follows that the genus of the $2$-core will always be equal to the genus of the original graph.

We are now ready to prove Theorem~\ref{sthm}.
It is well-known that
when $m = \frac{n}{2} + s$ for $s \gg n^{2/3}$,
the graph $G(n,m)$ will whp contain a unique largest component,
called the `giant' component,
and whp all other components will be planar
(see, for example, Theorem 6.15 of~\cite{bol};
see also \cite{bolrio1,bolrio2} for more recent work).
Hence, $g(n,m)$ is whp equal to the genus of the $2$-core of the giant component,
which we shall aim to determine via an application of Euler's formula.

This will require finding a bound for the number of faces in the $2$-core of the giant component,
and we shall proceed by first finding separate bounds for the number of such faces that are `small' and the number that are `large'.
The first part of this argument will involve bounding $Z(n,a(n))$ for a particular function $a(n) \to \infty$.

\begin{proof}[Proof of Theorem~\ref{sthm}]
Throughout this proof,
we shall consider only $i \in \mathbb{N} \cup \{0\}$,
and we shall let $Z(n,i)$, $Y(i)$, and $\lambda(i)$ be as given in the statement of Lemma~\ref{Poisson}.

Thus, for each $i$,
there exists a constant $N_{i}$ such that
\begin{displaymath}
\mathbb{P}[Z(n,i) \neq Y(i)] < \frac{1}{2^{i}} \quad \textrm{for all } n \geq N_{i}.
\end{displaymath}
Note that we may assume $N_{i+1} > N_{i}$ for all $i$,
and that $N_{0} = 0$.

Let us then define the function $b(n)$ by
\begin{displaymath}
b(n) =
\left\{ 
\begin{array}{lll}
0
& \textrm{for }
n < N_{1}, \\
1
& \textrm{for }
N_{1} \leq n < N_{2}, \\
2 
& \textrm{for } 
N_{2} \leq n < N_{3}, \\
3
& \textrm{for }
N_{3} \leq n < N_{4}, \\
\vdots &
\end{array} \right. 
\end{displaymath}
i.e. $b(n)=i$ for $N_{i} \leq n < N_{i+1}$ for all $i$
(note $b(n) \to \infty$ as $n \to \infty$).

Hence, for all $n$,
\begin{equation} \label{beqn}
\mathbb{P}[Z(n,i) \neq Y(i)] < \frac{1}{2^{i}} \quad \textrm{for all } i \leq b(n).
\end{equation}

Now let $c(r)$ denote
$\max \{ i : \lambda(i) \leq r \}$
(where we consider only $i \in \mathbb{N} \cup \{0\}$,
as always),
and note that $c(r)$ is well-defined for all $r \geq 0$,
since $\lambda (0) = 0$ and $\lambda (i) \to \infty$ as $i \to \infty$.

Let us then define the function $a(n)$ by
\begin{equation} \label{adefn}
a(n) = \min \left \{ b(n), \left\lfloor \left(0.05 \log \left( \frac{s^{3}}{n^{2}} \right) \right)^{1/2} \right\rfloor, c \left( \log \left( \frac{s^{3}}{n^{2}} \right) \right) \right \}.
\end{equation}
 
Since $a(n) \leq b(n)$,
\eqref{beqn} implies
\begin{displaymath} 
\mathbb{P}[Z(n,a(n)) \neq Y(a(n))] < \frac{1}{2^{a(n)}} \quad \textrm{for all } n.
\end{displaymath}
Hence,
noting that $a(n) \to \infty$ as $n \to \infty$,
we have
\begin{equation} \label{Zeqn}
Z(n,a(n)) = Y(a(n)) \quad \textrm{whp}.
\end{equation}

Now recall that $Y(i) \sim \textrm{Poi} (\lambda(i))$.
Hence, given any $\epsilon > 0$,
we certainly have $\mathbb{P}[Y(i) > 2 \lambda (i)] < \epsilon$
for all $i$ for which $\lambda(i)$ is sufficiently large.
Thus, since $\lambda(i) \to \infty$ as $i \to \infty$,
we actually have $\mathbb{P}[Y(i) > 2 \lambda (i)] < \epsilon$
for all sufficiently large $i$.
Hence, since $a(n) \to \infty$ as $n \to \infty$,
we have $\mathbb{P}[Y(a(n)) > 2 \lambda (a(n))] < \epsilon$
for all sufficiently large $n$,
i.e.~$Y(a(n)) \leq 2 \lambda (a(n))$ whp.

By definition of $a(n)$ (see~\eqref{adefn}) 
and the monotonicity of $\lambda$,
we have $\lambda (a(n)) \leq \log \left( \frac{s^{3}}{n^{2}} \right)$ for all $n$,
and hence $\lambda (a(n)) = o \left( \frac{s^{3}}{n^{2}} \right)$.
Thus, we have
\begin{equation} \label{Yeqn}
Y(a(n)) = o \left( \frac{s^{3}}{n^{2}} \right) \quad \textrm{whp}.
\end{equation}

Combining~\eqref{Zeqn} and~\eqref{Yeqn},
we hence obtain
\begin{displaymath}
Z(n,a(n)) = o \left( \frac{s^{3}}{n^{2}} \right) \quad \textrm{whp}. 
\end{displaymath} 
\\

For our function $a=a(n)$,
let us define a `short cycle' to be one with length less than $\frac{an}{s}$,
and let us use $C_{sh}(n)$ to denote
the number of short cycles in the giant component of $G(n,m)$
(or equivalently, in the $2$-core of the giant component).

It follows from Lemma~\ref{Fact8} that
whp all short cycles in $G(n,m)$ have no bad neighbours.
Also, by Lemma~\ref{Fact9},
whp all short cycles in $G(n,m)$ have a leaf neighbourhood smaller than $\frac{a^{2}n^{2}}{s^{2}}$ and less than $\frac{a^{2}n}{s}$ neighbours.
It is also the case that whp every short cycle in the giant component must have at least one good neighbour,
since otherwise (given that whp it has no bad neighbours)
the entire giant component would whp consist of just a cycle and its leaf neighbourhood,
and would thus be unicyclic
(which is well-known to be false whp ---
see, for example, Theorem 6.15 of~\cite{bol}).
Hence,
$C_{sh}(n) \leq Z(n,a(n))$ whp, 
and so we must have $C_{sh}(n) = o \left( \frac{s^{3}}{n^{2}} \right)$ whp. 

Now let us use $f_{sh}(n)$ to denote
the number of faces of length less than $\frac{an}{s}$ in any given embedding of the $2$-core of the giant component of $G(n,m)$.
Then it follows that we must have 
$f_{sh}(n) = o \left( \frac{s^{3}}{n^{2}} \right)$ whp too. \\

Let us similarly define
$f_{l}(n)$ to denote
the number of faces of length at least $\frac{an}{s}$ in any given embedding of the $2$-core of the giant component of $G(n,m)$.
Let us also use $v=v(n)$ and $e=e(n)$ to denote the number of vertices and edges,
respectively, in the $2$-core of the giant component.
Then
\begin{eqnarray*}
f_{l}(n) & \leq & \frac{2e}{an/s} \\
& = & \frac{2(v+(e-v))}{an/s} \\
& = & \frac{2((8+o(1))s^{2}/n + (16/3 + o(1)) s^{3}/n^{2})}{an/s}
\quad \textrm{whp} \\
& & \textrm{(by Theorem 4(i) of~\cite{luccyc} and Corollary 1 of~\cite{luccpt})} \\
& = & \frac{(16+o(1)) s^{2}/n}{an/s} \quad \textrm{whp \quad (since } s \ll n) \\
& = & (16 + o(1)) \frac{s^{3}}{a n^{2}} \quad \textrm{whp} \\
& = & o \left( \frac{s^{3}}{n^{2}} \right) \quad \textrm{whp \quad (since } a \to \infty). \\
\end{eqnarray*}

Thus, 
if we let $f=f(n)$ denote the total number of faces in any given embedding of the $2$-core of the giant component,
then we have 
$f = f_{sh}(n) + f_{l}(n) = o \left( \frac{s^{3}}{n^{2}} \right)$ whp.

Now recall that $g(n,m)$, the genus of $G(n,m)$, is whp equal to the genus of the giant component
(this was originally shown in Theorem 12(ii) of~\cite{luccpt}),
and that the latter is equal to the genus of the $2$-core of the giant component.

Hence,
whp 
\begin{eqnarray*}
\phantom{wwwwwwwwwqi}
g(n,m) & = & \frac{1}{2} (e-v-f+2) \textrm{ \quad (using Theorem~\ref{euler})} \\
& = & \frac{1}{2} \left( \left( \frac{16}{3} + o(1) \right) \frac{s^{3}}{n^{2}} - o \left( \frac{s^{3}}{n^{2}} \right) \right) \\
& & \textrm{(again using Corollary 1 of~\cite{luccpt} for $e-v$)} \\
& = & (1+o(1)) \frac{8s^{3}}{3n^{2}}.
\phantom{wwwwwwwwwwwwwwwwwwwwqq}
\qedhere
\end{eqnarray*}
\end{proof}

\section{Contiguity with random graphs on given surfaces} \label{contiguity}

One of our motivations for this paper comes from recent work concerning random graphs on given surfaces.
The typical properties of graphs with genus at most $g$ have been studied in~\cite{dks} and~\cite{kmsfull} for the case when $g$ is a fixed constant,
and questions have been posed on the likely behaviour when $g$ is allowed to grow with $n$.
Hence, in this section,
we shall discuss the contiguity (see Definition~\ref{contiguitydefn}) of such random graph models with $G(n)$ and $G(n,m)$.

We start with the definitions:

\begin{Definition} \label{Sgdefn}
We shall let 
\emph{$S_{g}(n)$}
denote a graph taken uniformly at random from
the set of all labelled graphs 
on $[n]$
with genus at most $g=g(n)$.

Similarly,
we shall let 
\emph{$S_{g}(n,m)$}
denote a graph taken uniformly at random from
the set of all labelled graphs 
on $[n]$
with exactly $m=m(n)$ edges and with genus at most $g=g(n)$.
\end{Definition}

\begin{Definition} \label{contiguitydefn}
We say that two random graph models $A(n)$ and $B(n)$ are \emph{contiguous} if,
for every property $P(n)$,
it is the case that $A(n)$ has property $P(n)$ whp
if and only if $B(n)$ has property $P(n)$ whp.
\end{Definition}

Recall that $G(n) = G_{n, \frac{1}{2}}$.
It hence follows from the work in~\cite{arch}
that $G(n)$ and $S_{g}(n)$ are certainly contiguous for any $g(n)$ satisfying
$g(n) \geq (1+\epsilon) \frac{n^{2}}{24}$
for any constant $\epsilon > 0$
(and even for some $\epsilon(n) =o(1)$),
since $G(n)$ will have genus 
$(1+o(1)) \frac{n^{2}}{24}$ whp.
Thus, the behaviour of $S_{g}(n)$ for such $g(n)$ follows immediately from known results on the behaviour of $G(n)$.

Conversely,
$G(n)$ and $S_{g}(n)$ are certainly not contiguous for any $g(n)$ satisfying
$g(n) \leq (1-\epsilon) \frac{n^{2}}{24}$
for any constant $\epsilon > 0$
(and also not for some $\epsilon(n) = o(1)$),
since there is then a discrepancy with respect to the property of having genus greater than $g(n)$
(observe that we would have $\mathbb{P}[G(n) \textrm{ has genus greater than } g(n)] \to 1$ as $n \to \infty$,
but that $\mathbb{P}[S_{g}(n)$ has genus greater than $g(n)] = 0$, by definition). 

Note that we could also obtain such a bound simply by considering the number of edges.
We know that the expected number of edges in $G(n)$ is 
$\frac{n(n-1)}{4}$,
and that $S_{g}(n)$ has at most $3n-6+6g(n)$ edges,
so this straight away implies that $G(n)$ and $S_{g}(n)$ must be non-contiguous for $g(n) < \frac{n^{2}}{24} - \frac{13n}{24} + 1$.

We may state our result as follows:

\begin{Theorem}
Let $\epsilon > 0$ be a fixed constant.
Then the random graphs $G(n)$ and $S_{g}(n)$ are contiguous
for $g(n) \geq (1+\epsilon) \frac{n^{2}}{24}$,
and are not contiguous for $g(n) \leq (1-\epsilon) \frac{n^{2}}{24}$.
\end{Theorem}

By the same arguments,
results on the contiguity of $G(n,m)$ and $S_{g}(n,m)$
for the various different regions of $m$ can now also be obtained,
using the bounds on $g(n,m)$ summarised in Table~\ref{table}:

\begin{Theorem}
Let $\epsilon > 0$ be a fixed constant,
let $m=m(n)$,
and let the function 
$\mu: \left( \frac{1}{2}, \infty \right) \to \mathbb{R},
\lambda \mapsto \mu (\lambda)$
be as defined in Theorem~\ref{lambdathm}.
Then the random graphs $G(n,m)$ and $S_{g}(n,m)$ are contiguous
for 
\begin{displaymath}
g(n) \geq
\left\{ 
{\renewcommand{\arraystretch}{1.5}
\begin{array}{lll}
(1+\epsilon) \frac{m}{6}
& \textrm{if }
m = \Theta \left( n^{2} \right), \\
(1+\epsilon) \frac{jm}{2(j+2)} 
& \textrm{if } 
n^{1 + \frac{1}{j+1}} \ll m = O \left( n^{1 + \frac{1}{j}} \right)
\textrm{ for } j \in \mathbb{N}, \\
\frac{m}{2} 
& \textrm{if }
n \ll m = n^{1+o(1)}, \\
(1+\epsilon) \mu(\lambda)m
& \textrm{if }
m \sim \lambda n
\textrm{ for } \lambda > \frac{1}{2}, \\
(1+\epsilon) \frac{8s^{3}}{3n^{2}} 
& \textrm{if } 
m = \frac{n}{2} + s
\textrm{ for } s>0 \textrm{ and } n^{2/3} \ll s \ll n,
\end{array}} \right. 
\end{displaymath}
and are not contiguous for 
\begin{displaymath}
g(n) \leq 
\left\{ 
{\renewcommand{\arraystretch}{1.5}
\begin{array}{lll}
(1-\epsilon) \frac{m}{6}
& \textrm{if }
m = \Theta \left( n^{2} \right), \\
(1-\epsilon) \frac{jm}{2(j+2)} 
& \textrm{if } 
n^{1 + \frac{1}{j}} \gg m = \Omega \left( n^{1 + \frac{1}{j+1}} \right)
\textrm{ for } j \in \mathbb{N}, \\
(1-\epsilon) \frac{m}{2} 
& \textrm{if }
n \ll m = n^{1+o(1)}, \\
(1-\epsilon) \mu(\lambda)m
& \textrm{if }
m \sim \lambda n
\textrm{ for } \lambda > \frac{1}{2}, \\
(1-\epsilon) \frac{8s^{3}}{3n^{2}} 
& \textrm{if } 
m = \frac{n}{2} + s
\textrm{ for } s>0 \textrm{ and } n^{2/3} \ll s \ll n.
\end{array}} \right. 
\end{displaymath}
\end{Theorem}

\section{The fragile genus: proof of Theorem~\ref{fragile}} \label{fragilesection}

In this section,
we shall prove Theorem~\ref{fragile},
which shows that the genus of a connected graph with bounded degree may well increase dramatically if a small number of random edges are added.

Our proof will involve contracting carefully chosen identically-sized pieces of the graph into super-vertices
(note that this cannot increase the genus),
and then showing that the uniform random graph induced by these super-vertices and the random edges
will whp be sufficiently dense for us to be able to apply Theorem~\ref{lambdathm}.

Before we begin,
we need to state the following useful decomposition:

\begin{Proposition} [\cite{krivnach}, Proposition 4.5] \label{krivnachprop}
Let $H$ be a connected graph with maximum degree at most $\Delta$.
Then for every $l \in \mathbb{N}$,
there exists $t \in \mathbb{N} \cup \{0\}$
and disjoint vertex sets
$V_{1}, V_{2}, \ldots, V_{t} \subset V(H)$
with the following properties: 
\begin{itemize}
\item[(i)] $l \Delta \leq |V_{i}| \leq l \Delta^{2}$ for all $i \in [t]$; 
\item[(ii)] $\sum_{i=1}^{t} |V_{i}| \geq |H| - l \Delta$; 
\item[(iii)] $H[V_{i}]$ is connected for all $i \in [t]$.
\end{itemize}
\end{Proposition}

We may now proceed:

\begin{proof}[Proof of Theorem~\ref{fragile}]
Note that adding an edge can only increase the genus by at most one,
so we certainly have
$g(G) \leq g(H) + k \leq 2 \max \left\{ g(H),k \right\}$.
Also, we clearly have $g(G) \geq g(H)$.
Hence, it just remains to show that
$g(G) = \Omega (k)$ whp.

Recall that
$R=R(n,k)$ is a random graph on $V(H)$ consisting of exactly $k$ edges chosen uniformly at random from $\binom{V(H)}{2}$.
We shall start by showing that we may assume that $k=O(n)$,
since we may otherwise just apply our results on the genus of $G(n,m)$ to $R$
(with $m=k$)
to obtain $g(R) = \Omega(k)$ whp.

As noted in Table~\ref{table}, by~\cite{rod} we have
\begin{equation} \label{new1}
g(R) = (1+o(1)) \frac{k}{6} \quad \textrm{whp}
\end{equation}
if $k = \Theta \left( n^2 \right)$.

If $6n \le k \ll n^2$,
then we may combine Euler's formula (Theorem~\ref{euler})
and Lemma~\ref{facelemma} to see
\begin{eqnarray}
g(R) & > & \frac{1}{2} (k-n-f(R)) 
\quad ( \textrm{by Euler's formula} ) \nonumber \\
& \ge & \frac{1}{2} \left( (1+o(1)) \frac{k}{3} - n \right) \quad \textrm{whp}
\quad \left( \textrm{by Lemma~\ref{facelemma}, since $k \ll n^2$ } \right) \nonumber \\
& \ge & (1+o(1)) \frac{k}{12} \quad \textrm{whp}
\quad \left( \textrm{since $n \le \frac{k}{6}$} \right). \label{new2}
\end{eqnarray}

Thus,
by~\eqref{new1} and~\eqref{new2},
we have $g(R) \ge (1+o(1)) \frac{k}{12}$ whp if $k \ge 6n$,
and so we may indeed assume throughout the remainder of the proof that $k = O(n)$.

Now let $l = \left \lceil \frac{3 \Delta n}{k} \right \rceil$
(note that $l = \Theta \left( \frac{n}{k} \right)$ since $k=O(n)$,
and also that $l=o(n)$),
and let the disjoint vertex sets
$V_{1}, V_{2}, \ldots, V_{t} \subset V(H) = V(G)$
satisfy the three properties stated in Proposition~\ref{krivnachprop}.
Note that
(by Properties~(i) and (ii) of Proposition~\ref{krivnachprop})
\begin{equation}
\frac{n - l \Delta}{l \Delta^{2}} \leq t \leq \frac{n}{l \Delta}, 
\label{teqn}
\end{equation}
and hence $t = \Theta (k)$.

From each of the sets $V_{1}, V_{2}, \ldots, V_{t}$,
let us select a subset $U_i \subset V_i$
with $|U_i| = s := \min_{j} |V_{j}|$
in such a way that each $U_i$ spans a connected subgraph
(this is possible,
by Property~(iii) of Proposition~\ref{krivnachprop}).

Next, let us define an auxiliary random graph $\Gamma$
with vertex set $[t]$,
where two vertices $i,j \in [t]$ are connected by an edge 
if and only if there is an edge of $R$ going between $U_{i}$ and $U_{j}$
(see Figure~\ref{Gammafig},
where thick lines denote the edges of $R$
--- note in particular that this example has no edge in $\Gamma$ between vertex $1$ and vertex $3$,
as there is no edge in $G$ between $U_{1}$ and $U_{3}$).

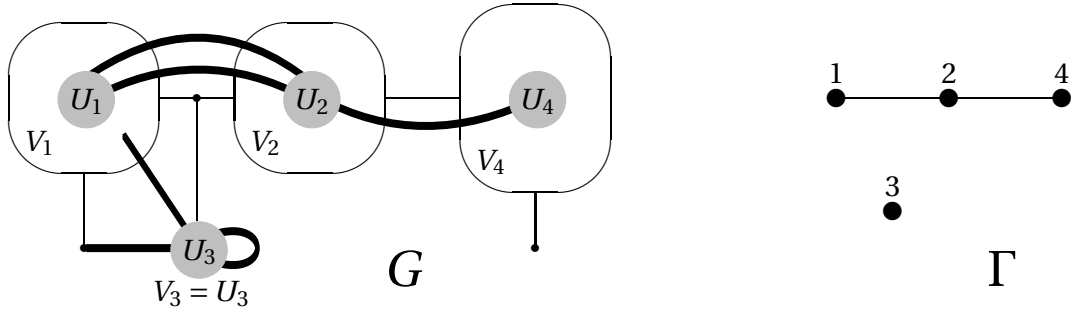
\begin{figure} [ht]
\setlength{\unitlength}{1cm}
\begin{picture}(14.125,3.95)(-1,-0.7)

\put(0,0){\circle*{0.1}}
\put(1.5,2){\circle*{0.1}}
\put(6,0){\circle*{0.1}}

\put(0.9,-0.7){$V_{3}=U_{3}$}

\put(0,2){\oval(2,2)}
\put(-0.8,1.3){$V_{1}$}

\put(3,2){\oval(2,2)}
\put(2.2,1.3){$V_{2}$}

\put(6,2){\oval(2,2.5)}
\put(5.2,1.05){$V_{4}$}

\put(1,2){\line(1,0){1}}
\put(4,2){\line(1,0){1}}
\put(6,0){\line(0,1){0.75}}
\put(0,0){\line(0,1){1}}
\put(1.5,2){\line(0,-1){1.625}}

\put(4,-0.5){\LARGE{$G$}}

\put(-0.4375,-0.55){
\begin{tikzpicture} 

\draw[black, ultra thick]
(0,0) -- (1.2,0);
\draw[black, ultra thick]
(0,0.05) -- (1.2,0.05);
\draw[black, ultra thick]
(1.6,0.15) .. controls (2.5,0.5) and (2.5,-0.5) .. (1.6,-0.15);
\draw[black, ultra thick]
(1.6,0.2) .. controls (2.5,0.55) and (2.5,-0.45) .. (1.6,-0.1);
\draw[black, ultra thick]
(1.6,0.175) .. controls (2.5,0.525) and (2.5,-0.475) .. (1.6,-0.125);
\draw[black, ultra thick]
(1.4,0.15) -- (0.5,1.5);
\draw[black, ultra thick]
(1.4,0.2) -- (0.5,1.55);
\draw[black, ultra thick]
(1.4,0.175) -- (0.5,1.525);
\draw[black, ultra thick]
(3.1,2) .. controls (4,1.5) and (5,1.5) .. (6,2);
\draw[black, ultra thick]
(3.1,2.05) .. controls (4,1.55) and (5,1.55) .. (6,2.05);
\draw[black, ultra thick]
(3.1,2.025) .. controls (4,1.525) and (5,1.525) .. (6,2.025);
\draw[black, ultra thick]
(0.1,2) .. controls (1,2.5) and (2,2.5) .. (2.9,2);
\draw[black, ultra thick]
(0.1,2.05) .. controls (1,2.55) and (2,2.55) .. (2.9,2.05);
\draw[black, ultra thick]
(0.1,2.025) .. controls (1,2.525) and (2,2.525) .. (2.9,2.025);
\draw[black, ultra thick]
(-0.1,2.2) .. controls (1,3) and (2,3) .. (3,2.2);
\draw[black, ultra thick]
(-0.1,2.25) .. controls (1,3.05) and (2,3.05) .. (3,2.25);
\draw[black, ultra thick]
(-0.1,2.225) .. controls (1,3.025) and (2,3.025) .. (3,2.225);

\filldraw [lightgray] (0,2) circle (0.375);
\filldraw [lightgray] (1.5,0) circle (0.375);
\filldraw [lightgray] (3,2) circle (0.375);
\filldraw [lightgray] (6,2) circle (0.375);

\end{tikzpicture}
}

\put(-0.2,1.85){$U_{1}$}
\put(2.8,1.85){$U_{2}$}
\put(5.8,1.85){$U_{4}$}
\put(1.3,-0.15){$U_{3}$}

\put(10,2){\circle*{0.25}}
\put(11.5,2){\circle*{0.25}}
\put(13,2){\circle*{0.25}}
\put(10.75,0.5){\circle*{0.25}}

\put(10,2){\line(1,0){3}}

\put(9.9,2.2){$1$}
\put(11.4,2.2){$2$}
\put(12.9,2.2){$4$}
\put(10.65,0.7){$3$}

\put(12,-0.5){\LARGE{$\Gamma$}}

\end{picture}

\caption{A graph $G$ and the corresponding graph $\Gamma$.} \label{Gammafig}
\end{figure}

Observe that $\Gamma$ is a minor of $G$,
and hence $g(G) \geq g(\Gamma)$.
Thus (recalling that $t = \Theta (k)$), 
it will suffice to show that 
\begin{equation}
g(\Gamma) = \Omega(t) \quad \textrm{whp.}
\label{aim}
\end{equation}

Let us consider the edges of $R$ one-by-one
(in a random order),
and let us call an edge `good' if both
\begin{itemize}
\item[(a)] it lies between a vertex of $U_{i}$ and a vertex of $U_{j}$ for $i \neq j$; 
\item[(b)] no previous edges of $R$ lie between these same two sets $U_{i}$ and $U_{j}$. 
\end{itemize}
Note that $e(\Gamma)$ is then the total number of good edges of $R$.

Observe that the probability that the $r$th edge of $R$ is good is at least
\begin{displaymath}
\frac{\left( \binom{t}{2}-(r-1) \right) s^{2}}{\binom{n}{2}},
\end{displaymath}
since there are at least
$\binom{t}{2} - (r-1)$ ways to choose a pair $U_{i}, U_{j}$
which do not already have an edge of $R$ between them,
and then $s$ ways to choose a vertex from $U_{i}$,
and $s$ ways to choose a vertex from $U_{j}$.

Furthermore, we have
\begin{eqnarray*}
\frac{\left( \binom{t}{2}-(r-1) \right) s^{2}}{\binom{n}{2}} 
& \geq & \frac{\left( \binom{t}{2}-(k-1) \right) s^{2}}{\binom{n}{2}} \\
& = & (1+o(1)) \frac{\binom{t}{2} s^{2}}{\binom{n}{2}} 
\textrm{ \quad (since $t=\Theta(k)$ and $k \to \infty$)} \\
& = & (1+o(1)) \frac{t^{2} s^{2}}{n^{2}} \\
& \geq & (1+o(1)) \frac{(n-l \Delta)^{2} s^{2}}{l^{2} \Delta^{4} n^{2}} 
\textrm{ \quad (by~\eqref{teqn})} \\
& = & (1+o(1)) \frac{s^{2}}{l^{2} \Delta^{4}} 
\textrm{ \quad (since $l=o(n)$)} \\
& \geq & (1+o(1)) \frac{1}{\Delta^{2}} 
\textrm{ \quad (by Property (i) of Proposition~\ref{krivnachprop})} \\
& \geq & \frac{1}{2 \Delta^{2}}
\textrm{ \quad (for large } n).
\end{eqnarray*}

Thus, the number of good edges of $R$
(i.e.~$e(\Gamma)$)
stochastically dominates a random variable with distribution
Bin$\left(k,\frac{1}{2\Delta^{2}}\right)$,
and is therefore at least 
$\frac{k}{(2+\delta)\Delta^{2}}$ whp
for any $\delta > 0$.
Now recall that
$t \leq \frac{n}{l\Delta} \leq \frac{k}{3 \Delta^{2}}$
(since $l \geq \frac{3 \Delta n}{k}$),
and so we find that whp
$e(\Gamma) \geq t$.
Hence, we may assume that $\Gamma$ has at least $t$ edges.

Now let $\Gamma^{*}$ be the random graph formed by considering just the first $t$ edges of $\Gamma$
(i.e.~the first $t$ good edges of $R$).
Then, since each set $U_{i}$ had exactly the same number of vertices,
the graph $\Gamma^{*}$ is in fact a \emph{uniform} random graph with $t$ vertices and $t$ edges.
Thus, by Theorem~\ref{lambdathm},
we have $g(\Gamma^{*}) = \Theta(t)$ whp,
and so $g(\Gamma) = \Omega(t)$ whp,
fulfilling~\eqref{aim}.
\end{proof}

Note that Theorem~\ref{fragile} implies the remarkable fact that whp $G = H \cup R$ will have $\Omega (n)$ genus
even if $H$ is a planar graph
and $k=\epsilon n$ for some very small (but positive) $\epsilon$!
We thus call this the `fragile genus' property.

\section{Discussion} \label{discussion}

In this paper,
we have investigated the genus $g(n,m)$ of the Erd\H{o}s-R\'enyi random graph
\linebreak[4]
$G(n,m)$,
showing how this is affected by changes in $m$.
We have then also examined how the genus of a given base graph grows when random edges are added.
In this section,
we shall aim to give some insight into our results,
as well as discussing various remaining open problems.

Recall that we earlier observed that the ratio of $g(n,m)$ to $m$ increases from $0$ to $\frac{1}{2}$ until $m$ becomes superlinear in $n$,
after which it then decreases from $\frac{1}{2}$ to $\frac{1}{6}$
(see Table~\ref{table}).
Having seen the proofs for these results,
we are now in a position to provide an explanation for this behaviour.

The starting point is to note that Euler's formula gives
\begin{displaymath}
\frac{g}{m} = \frac{1}{2} \left( 1 - \frac{f}{m} + \frac{\kappa}{m} - \frac{n}{m} + \frac{1}{m} \right),
\end{displaymath}
and so we need to consider the terms $\frac{f}{m}$, $\frac{\kappa}{m}$, and $\frac{n}{m}$.

When $m$ is only linear in $n$,
we have
$\frac{f}{m} = o(1)$ whp (see Corollary~\ref{facecor}),
and so the significant terms are $\frac{\kappa}{m}$ and $\frac{n}{m}$.

For the subcritical case
$m \sim \lambda n$ for $\lambda \leq \frac{1}{2}$,
we have 
$\kappa = n - m + o(n)$ (see Corollary~\ref{bol612} and Lemma~\ref{bolremark}),
and so $\frac{g}{m}$ is around $0$.

For $\lambda > \frac{1}{2}$,
$\kappa$ now decreases more slowly with $m$ (see Lemma~\ref{derivlemma}),
and so $\frac{\kappa - n}{m}$ increases
(i.e.~gets closer to $0$),
hence $\frac{g}{m}$ increases.
By the time $\lambda$ is very large,
$\frac{\kappa}{m}$ and $\frac{n}{m}$ will both be very small,
and so $\frac{g}{m}$ is close to $\frac{1}{2}$.

Finally, for the superlinear case $m \gg n$,
we have $\frac{n}{m} = o(1)$ and $\frac{\kappa}{m} \leq \frac{n}{m} = o(1)$,
but $\frac{f}{m}$ now grows beyond $o(1)$,
and so $\frac{g}{m}$ decreases.
However, due to the well-known inequality $m \leq 3n-6+6g$,
we still have a lower bound 
$\frac{g}{m} \geq \frac{1}{6} + \frac{2-n}{2m} = \frac{1}{6} + o(1)$.

Let us now conclude this section by highlighting some of the remaining open problems.

For the topic of $g(n,m)$,
we have provided a thorough description with accuracy of $1+o(1)$,
but it would be interesting to know whether even more precise results can be obtained.
In particular,
it would be nice to determine the exact behaviour of $g(n,m)$
in the region when $m = \Theta \left( n^{1+\frac{1}{j}} \right)$,
for which we currently have the bounds
$(1+o(1)) \frac{(j-1)m}{2(j+1)} \leq g(n,m) \leq \frac{jm}{2(j+2)}$ whp from~\cite{rod}.
Also, would it be possible to extend our whp results to exponentially whp?

We then showed that adding $k$ edges
(for $k \to \infty$)
to any given connected graph $H$ with bounded maximum degree
will whp result in a supergraph with genus 
$\Theta \left( \max \left\{ g(H), k \right\} \right)$.
Again, it would be interesting to know whether this result can be improved further, for instance to accuracy of $1+o(1)$.

Next,
we recall that one of our motivations for this paper came from connections with random graphs on given surfaces.
We have observed that such graphs
are contiguous with $G(n)$ and $G(n,m)$ beyond certain values of $g$,
but it would be fascinating to ascertain how greatly the properties of these models differ from $G(n)$ and $G(n,m)$ when $g$ is just below the contiguity threshold.

Finally,
let us mention that it seems that analogous results to those in this paper
can also be obtained for non-orientable surfaces.

\section*{Acknowledgements}

A major part of this work was performed when the third author visited the Institute of Discrete Mathematics at TU Graz.
We would like to thank Philipp Spr\"ussel for helpful discussions,
and we are also grateful to the referees of both the full version
and the extended abstract of this paper.

\end{document}